\documentclass[a4paper,11pt]{article}
\usepackage[utf8]{inputenc}
\usepackage{a4wide}
\usepackage{algorithm}
\usepackage{algorithmic}
\usepackage{latexsym,amsfonts,amsmath,amssymb,mathrsfs,url,amsthm}
\usepackage{mathtools}
\usepackage{dsfont}
\usepackage{color,graphicx}
\usepackage{lipsum}
\usepackage{subcaption}
\usepackage{hyperref}
\usepackage{xcolor}
\newtheorem{theorem}{Theorem}[section]
\newtheorem{lemma}[theorem]{Lemma}

\newtheorem{remark}[theorem]{Remark}
\newtheorem{definition}[theorem]{Definition}

\mathtoolsset{showonlyrefs}

\newcommand{\rd}{\, \mathrm{d}}

\newcommand{\wor}{\mathrm{wor}}

\newcommand{\sob}{\mathrm{sob}}
\newcommand{\sym}{\mathrm{sym}}
\newcommand{\bszero}{\boldsymbol{0}}
\newcommand{\bsone}{\boldsymbol{1}}
\newcommand{\bsa}{\boldsymbol{a}}
\newcommand{\bsh}{\boldsymbol{h}}

\newcommand{\bsk}{\boldsymbol{k}}
\newcommand{\bsl}{\boldsymbol{\ell}}
\newcommand{\bsx}{\boldsymbol{x}}
\newcommand{\bsy}{\boldsymbol{y}}
\newcommand{\bsz}{\boldsymbol{z}}
\newcommand{\bsgamma}{\boldsymbol{\gamma}}
\newcommand{\bsGamma}{\boldsymbol{\Gamma}}
\newcommand{\bssigma}{\boldsymbol{\sigma}}
\newcommand{\mi}{\mathrm{i}}
\newcommand{\EE}{\mathbb{E}}
\newcommand{\NN}{\mathbb{N}}
\newcommand{\PP}{\mathbb{P}}
\newcommand{\RR}{\mathbb{R}}
\newcommand{\ZZ}{\mathbb{Z}}
\newcommand{\Acal}{\mathcal{A}}
\newcommand{\Bcal}{\mathcal{B}}
\newcommand{\Ocal}{\mathcal{O}}
\newcommand{\Pcal}{\mathcal{P}}
\DeclareMathOperator{\supp}{supp}

\allowdisplaybreaks

\title{Multivariate integration and approximation in weighted Sobolev spaces of low fractional smoothness\thanks{The work of T.G.\ was supported by JSPS KAKENHI Grant Number JP26K22266.}}
\author{Mou Cai\thanks{Graduate School of Engineering, The University of Tokyo, 7-3-1 Hongo, Bunkyo-ku, Tokyo 113-8656, Japan (\url{caimoumou@g.ecc.u-tokyo.ac.jp}; \url{goda@frcer.t.u-tokyo.ac.jp})} \and Takashi Goda\footnotemark[1]}
\date{\today}

\begin{document}

\maketitle

\begin{abstract}
The weighted half-period cosine space has often been employed in the theory of quasi-Monte Carlo methods for multivariate integration and approximation of non-periodic functions. For integer-order smoothness, its norm equivalence to certain weighted unanchored Sobolev spaces has been established in the literature. In this work, we extend this equivalence to fractional-order smoothness up to $2$. By introducing an explicit representation via Slobodeckij-type seminorms, we prove a norm equivalence between the half-period cosine spaces and the corresponding weighted unanchored Sobolev spaces. Our Sobolev norm representation clarifies how the fractional regularity dictates the presence or absence of boundary constraints and (non-)periodic structures. Furthermore, we investigate the limiting behavior of these fractional spaces as the smoothness parameter approaches integer boundaries, establishing a continuous bridge to the classical integer-order Sobolev spaces. These equivalence results enable us to transfer the near-optimal error bounds and tractability results for multivariate integration and function approximation from the half-period cosine settings to our newly introduced fractional Sobolev spaces.
\end{abstract}


%%%%%%%%%%%%%%%%%%%%%%%%%%%%%%%%%%%%%%%%%%%%%%%%%%
%%%%%%%%%%%%%%%%%%%%%%%%%%%%%%%%%%%%%%%%%%%%%%%%%%
\section{Introduction}\label{sec:intro}
Numerical integration and function approximation in high dimensions are fundamental computational tasks arising in various fields, including financial engineering \cite{G04,L09}, uncertainty quantification \cite{DGIP21,S15}, and the numerical solution of partial differential equations with random coefficients \cite{KN16,SG11}. Over the past decades, quasi-Monte Carlo (QMC) methods based on rank-1 lattice point sets and (higher-order) digital nets have been studied intensively in mitigating the curse of dimensionality and attaining tractability in the sense of Information-Based Complexity \cite{DKP22,DKS13,DP10,N92,NW08,NW10,SJ94}. To analyze the convergence behavior of QMC-based algorithms, the target functions are typically modeled in reproducing kernel Hilbert spaces (RKHS), where smoothness and weights play important roles in error decay and dependence on the dimension $d$, respectively. Here, the weights are assigned to each subset of variables, modeling the relative importance of the corresponding coordinate directions within the norm of the space \cite{H98,SW98}.

Among various classes of RKHS used in QMC theory, classical Fourier-based spaces, such as the weighted Korobov space, have provided a central framework for periodic functions on $[0,1]^d$. This prominence is partly due to their affinity with rank-1 lattice point sets; on the Korobov space, the worst-case integration error of lattice rules can be expressed in an explicitly computable form, enabling the efficient search for good generating vectors via the component-by-component (CBC) construction \cite{DGS22,K03,NC06,SR02}. Furthermore, rank-1 lattice-based algorithms within the Korobov space have been extensively investigated for multivariate function approximation. To date, various approaches have been developed, including single rank-1 lattice algorithms \cite{CGK26,KSW06,KWW09,ZKH09,ZLH06}, multiple lattice algorithms \cite{CG26,K18,K19,KV19}, median lattice algorithms \cite{PCDGK26,PGK25,PKG25}, lattice with multiple shift algorithms \cite{CDG25,DD26}, and subsampling algorithms \cite{BGKS26}. Although the non-optimality of single rank-1 lattice algorithms in the current context was proven in \cite[Section~3]{BKUV17}, essentially the same results were already available in the classical literature \cite{HW81,K63,S60}. This is why many variants as above have been proposed to address the (near-)optimality in terms of the convergence rate and/or tractability.

However, most practical target functions, such as those appearing in the aforementioned applications, are inherently non-periodic. If one forces a periodic structure onto a non-periodic function via artificial periodic continuation, severe discontinuities or boundary singularities are typically introduced at the domain boundaries. In the context of lattice-based QMC integration and approximation, such artificial boundary singularities dramatically deteriorate the convergence behavior. To circumvent this issue without losing the computational advantages of spectral representations, the weighted half-period cosine space, denoted by $H_{\alpha,\bsgamma,d}^{\cos}$, was introduced as a powerful alternative that bypasses the need for periodic assumptions \cite{DNP14}. By employing a complete orthonormal basis composed of half-period cosines, this space successfully captures the non-periodic nature of functions over $[0,1]^d$, while maintaining an explicit spectral-side representation that inherits the excellent compatibility with rank-1 lattice point sets when combined with the tent transformation $\varphi(x)=1-|2x-1|$.

To understand the physical meaning of the half-period cosine space, it is essential to establish its connection to classical Sobolev spaces defined via weak derivatives in the physical domain. For integer-order smoothness (e.g., $\alpha = 1$ or $\alpha = 2$), it has been shown in \cite{DNP14,GSY19} that the half-period cosine spaces are norm-equivalent to certain weighted unanchored Sobolev spaces. This equivalence bridges the gap between the spectral-side representation, which is ideal for algorithm construction and error analysis, and the physical-side representation, which directly reflects the intrinsic differentiability and boundary constraints of functions. However, a significant gap remains: the existing equivalence results are strictly confined to integer-order smoothness. In this paper, we address this problem by focusing on the case of fractional low-order smoothness $1/2<\alpha<2$. In passing, the Besov regularity of functions in $H_{\alpha,\gamma,d}^{\cos}$ with fractional $\alpha<3/2$ has recently been investigated in \cite{ST25}.

The remainder of this paper is organized as follows. Section~\ref{sec:pre} provides a brief review of the weighted half-period cosine space $H_{\alpha,\bsgamma,d}^{\cos}$ and rank-1 lattice point sets. In Section~\ref{sec:equivalence}, we analyze the cases $1/2<\alpha<1$ and $1<\alpha<2$ separately to establish the norm equivalence between the half-period cosine spaces and certain fractional Sobolev spaces involving Slobodeckij-type seminorms. This characterization explicitly clarifies how the fractional regularity dictates the presence or absence of boundary constraints and (non-)periodic structures. We also investigate the limiting behavior of these fractional spaces as the smoothness parameter approaches the integer boundaries. Finally, Sections~\ref{sec:integration} and \ref{sec:approximation} are devoted to the study of numerical integration and function approximation, respectively, in our novel fractional Sobolev spaces, achieved by transferring the results from the half-period cosine settings.

%%%%%%%%%%%%%%%%%%%%%%%%%%%%%%%%%%%%%%%%%%%%%%%%%%
%%%%%%%%%%%%%%%%%%%%%%%%%%%%%%%%%%%%%%%%%%%%%%%%%%
\section{Preliminaries}\label{sec:pre}
Throughout this paper, $\ZZ$, $\NN$, and $\NN_0$ denote the sets of integers, positive integers, and non-negative integers, respectively. For a multi-index $\bsk = (k_1, \ldots, k_d) \in \ZZ^d$ or $\bsk \in \NN_0^d$, we write $\supp(\bsk) \coloneqq \{1 \le j \le d \mid k_j \ne 0\}$ for its support and $|\bsk|_0 \coloneqq |\supp(\bsk)|$ for the number of non-zero components in $\bsk$. For a subset $u \subseteq \{1,\ldots,d\}$, we denote its cardinality by $|u|$ and its complement by $-u \coloneqq \{1,\ldots,d\}\setminus u$.

%%%%%%%%%%%%%%%%%%%%%%%%%%%%%%%%%%%%%%%%%%%%%%%%%%
\subsection{Weighted half-period cosine space}
For $k \in \NN_0$, we define the one-dimensional half-period cosine basis functions on $[0,1]$ as
\[ \phi_k(x) \coloneqq \begin{cases}
        1 & \text{if } k=0,\\
        \sqrt{2}\cos(k\pi x) & \text{if } k \in \NN.
    \end{cases} \]
For the $d$-dimensional case, we define the tensor product basis functions by
\[    \phi_{\bsk}(\bsx) \coloneqq \prod_{j=1}^{d} \phi_{k_j}(x_j)  \]
for $\bsk=(k_1,\ldots,k_d)\in \NN_0^d$ and $\bsx=(x_1,\ldots,x_d)\in [0,1]^d$. It is known that the set $\{\phi_{\bsk} \mid \bsk \in \NN_0^d\}$ forms a complete orthonormal basis for $L_2([0,1]^d)$ (see \cite{DNP14}; see also \cite{IN08} for a discussion of the corresponding basis over the interval $[-1,1]$). Thus, any function $f \in L_2([0,1]^d)$ can be represented by its half-period cosine expansion:
\[  f(\bsx) = \sum_{\bsk \in \NN_0^d} \widetilde{f}(\bsk) \phi_{\bsk}(\bsx),    \]
where the equality holds in the $L_2$-sense, and $\widetilde{f}(\bsk)$ denotes the $\bsk$-th half-period cosine coefficient of $f$, given by
\[  \widetilde{f}(\bsk) \coloneqq \int_{[0,1]^d} f(\bsx) \phi_{\bsk}(\bsx) \rd \bsx. \]

For a smoothness parameter $\alpha>1/2$ and a coordinate-weight parameter $\gamma>0$, we define the one-dimensional weight function
\[ r_{\alpha,\gamma}(k) \coloneqq \begin{cases}
    1 & \text{if } k=0, \\ \gamma^{-1}k^{\alpha} & \text{if } k \in \NN.
\end{cases}\]
For a multi-index $\bsk\in \NN_0^d$, we define its $d$-dimensional counterpart by
\[ r_{\alpha,\bsgamma}(\bsk) \coloneqq \prod_{j=1}^{d}r_{\alpha,\gamma_j}(k_j), \]
where $\bsgamma=(\gamma_1,\gamma_2,\ldots)$ is a sequence of positive weights. Then, the weighted half-period cosine space, denoted by $H_{\alpha,\bsgamma,d}^{\cos}$, is a reproducing kernel Hilbert space equipped with the reproducing kernel
\[ K_{\alpha,\bsgamma,d}^{\cos}(\bsx,\bsy) \coloneqq \sum_{\bsk\in \NN_0^d}\frac{\phi_{\bsk}(\bsx)\phi_{\bsk}(\bsy)}{r_{\alpha,\bsgamma}^2(\bsk)},\]
and the inner product
\[ \langle f,g\rangle_{\alpha,\bsgamma,d}^{\cos} \coloneqq \sum_{\bsk\in \NN_0^d}\widetilde{f}(\bsk)\widetilde{g}(\bsk) r_{\alpha,\bsgamma}^2(\bsk). \]
The induced norm is denoted as $\|f\|_{\alpha,\bsgamma,d}^{\cos} \coloneqq \sqrt{\langle f,f\rangle_{\alpha,\bsgamma,d}^{\cos}}$.

%%%%%%%%%%%%%%%%%%%%%%%%%%%%%%%%%%%%%%%%%%%%%%%%%%
\subsection{Rank-1 lattice rules}\label{subsec:lattice}
Next, we introduce the sampling points used in our numerical integration and function approximation algorithms.

\begin{definition}
    Let $n\in \NN$ with $n\ge 2$ be the number of points, and let $\bsz=(z_1,\ldots,z_d)\in \{0,\ldots,n-1\}^d$ be a generating vector. The rank-1 lattice point set $P_{n,\bsz}$ is defined by
    \[ P_{n,\bsz} \coloneqq \left\{ \bsy_i = \left( \left\{\frac{iz_1}{n}\right\} ,\ldots, \left\{\frac{iz_d}{n}\right\}\right) \mathrel{\bigg|} 0\le i<n \right\}, \]
    where $\{x\} \coloneqq x-\lfloor x\rfloor$ denotes the fractional part of a real number $x$.
\end{definition}

If $n$ is prime and $z_1\ne 0$, the rank-1 lattice point set can be expressed as the intersection $P_{n,\bsz}=\Lambda\cap [0,1)^d$, where $\Lambda$ denotes the Euclidean lattice 
\[ \Lambda \coloneqq T\ZZ^d = \left\{ T\bsk \mid \bsk\in \ZZ^d\right\} \]
with the generator matrix
\[ T \coloneqq \begin{pmatrix} 1/n & 0 & \dots & 0 \\ z_1^{-1}z_2/n & 1 & \dots & 0 \\ \vdots & \vdots & \ddots & \vdots \\ z_1^{-1}z_d/n & 0 & \dots & 1 \end{pmatrix}, \]
and where $z_1^{-1} \in \ZZ_n$ is the multiplicative inverse of $z_1$ modulo $n$. Since the Euclidean lattice can equivalently be represented as
\[ \Lambda = \frac{1}{n}\bsz\ZZ + \ZZ^d, \]
the corresponding dual lattice is straightforwardly characterized as
\begin{align*}
    P_{n,\bsz}^{\perp} \coloneqq \Lambda^{\perp} &= \{ \bsk \in \RR^d \mid \bsk \cdot \bsy \in \ZZ \text{ for all } \bsy \in \Lambda \} \\
    &= \{ \bsk \in \ZZ^d \mid \bsk \cdot \bsz \equiv 0 \pmod n \}.
\end{align*}

To deal with non-periodic functions in the weighted half-period cosine spaces $H_{\alpha,\bsgamma,d}^{\cos}$, we apply the tent transformation to the rank-1 lattice point set. The one-dimensional tent transformation $\varphi \colon [0,1] \to [0,1]$ is defined as
\[ \varphi(x) \coloneqq 1-|2x-1| = \begin{cases}
    2x & \text{if } 0\le x < 1/2,\\
    2-2x & \text{if } 1/2 \le x \le 1.
\end{cases} \]
For a $d$-dimensional vector $\bsx = (x_1, \ldots, x_d)$, the transformation is applied component-wise, i.e., $\varphi(\bsx) \coloneqq (\varphi(x_1), \ldots, \varphi(x_d))$. Then, the tent-transformed rank-1 lattice point set is
\begin{align}\label{eq:tent_lattice}
    P^{\varphi}_{n,\bsz} \coloneqq \varphi(P_{n,\bsz}) = \left\{ \bsx_i = \varphi(\bsy_i) \mid 0 \le i < n \right\}. 
\end{align}

%%%%%%%%%%%%%%%%%%%%%%%%%%%%%%%%%%%%%%%%%%%%%%%%%%
%%%%%%%%%%%%%%%%%%%%%%%%%%%%%%%%%%%%%%%%%%%%%%%%%%
\section{Function spaces and norm equivalence}\label{sec:equivalence}

Here, as one of the main results of this paper, we prove that the half-period cosine spaces with smoothness $1/2<\alpha<2$ are norm-equivalent to certain fractional Sobolev spaces involving Slobodeckij-type seminorms. Although our analysis below can be extended to arbitrary higher-order fractional smoothness $\alpha > 2$ by considering higher-order weak derivatives and appropriate boundary conditions, we restrict our focus to the range $1/2 < \alpha < 2$ in this paper. We refer the reader to \cite{AF03,L23,M85} for comprehensive information on the general theory of fractional Sobolev spaces. 

To formalize our discussion, we first recall the concept of norm equivalence. Let $H(K_1)$ and $H(K_2)$ be two reproducing kernel Hilbert spaces defined on the same domain, equipped with norms $\|\cdot\|_{K_1}$ and $\|\cdot\|_{K_2}$, respectively. We say that $H(K_1)$ and $H(K_2)$ are \emph{norm-equivalent} if they coincide as sets, i.e., $H(K_1) = H(K_2)$, and if there exist positive constants $c$ and $C$ such that
\[ c \|f\|_{K_2} \le \|f\|_{K_1} \le C \|f\|_{K_2} \quad \text{for all } f \in H(K_1). \]

Before moving on to the fractional setting, we review the known results on the norm equivalence for integer-order smoothness from \cite{DNP14,GSY19}.

%%%%%%%%%%%%%%%%%%%%%%%%%%%%%%%%%%%%%%%%%%%%%%%%%%
\subsection{Known results for integer-order smoothness}\label{subsec:known_equivalence}

For simplicity, let us focus on the univariate case; the multivariate counterpart can be straightforwardly derived from the tensor product structure of the spaces. For $\alpha\in \NN$ and $\gamma>0$, let $H_{\alpha,\gamma,1}^{\mathrm{sob}}$ be the reproducing kernel Hilbert space equipped with the reproducing kernel
\[ 
    K_{\alpha,\gamma,1}^{\mathrm{sob}}(x, y) \coloneqq 1 + \gamma^2 \left( \sum_{\tau=1}^{\alpha} \frac{B_\tau(x) B_\tau(y)}{(\tau!)^2} + (-1)^{\alpha+1} \frac{B_{2\alpha}(|x-y|)}{(2\alpha)!} \right),
\]
and the inner product
\begin{align*} 
    \langle f,g \rangle_{\alpha,\gamma,1}^{\mathrm{sob}} & \coloneqq \int_0^1 f(x) \rd x \int_0^1 g(x) \rd x \\
    & \quad + \frac{1}{\gamma^2} \left( \sum_{\tau=1}^{\alpha-1} \int_0^1 f^{(\tau)}(x) \rd x \int_0^1 g^{(\tau)}(x) \rd x + \int_0^1 f^{(\alpha)}(x) g^{(\alpha)}(x) \rd x \right). 
\end{align*}
Here, $B_{\tau}$ denotes the Bernoulli polynomial of degree $\tau$, and $f^{(\tau)}$ denotes the $\tau$-th weak derivative of $f$. This space is commonly referred to as the \emph{unanchored Sobolev space} in the QMC literature \cite{DKS13}. 

Regarding the norm equivalence for the half-period cosine space, \cite{DNP14} established the result for $\alpha=1$, whereas \cite{GSY19} addressed the general integer case $\alpha \ge 1$. In particular, it was proved in \cite[Lemma~1]{GSY19} that the half-period cosine space $H_{\alpha,\gamma,1}^{\cos}$ is norm-equivalent to the unanchored Sobolev space with vanishing boundary conditions on odd-order derivatives, defined by
\[ 
    H_{\alpha,\gamma,1}^{\mathrm{sob}(\mathrm{odd-bdry0})} \coloneqq \left\{ f \in H_{\alpha,\gamma,1}^{\mathrm{sob}} \mathrel{\bigg|} f^{(\tau)}(0) = f^{(\tau)}(1) = 0 \text{ for all odd } \tau < \alpha \right\}. 
\]
Thus, in the special case of $\alpha=1$, the half-period cosine space $H_{1,\gamma,1}^{\cos}$ is norm-equivalent to the full unanchored Sobolev space $H_{1,\gamma,1}^{\mathrm{sob}}$, recovering the earlier result in \cite[Theorem~1]{DNP14}. 

In what follows, we focus on the fractional cases $1/2<\alpha<1$ and $1<\alpha<2$, and introduce the respective Sobolev-type spaces that are norm-equivalent to the half-period cosine spaces in these ranges.

%%%%%%%%%%%%%%%%%%%%%%%%%%%%%%%%%%%%%%%%%%%%%%%%%%
\subsection{The case of $1/2 < \alpha < 1$}
For the lower fractional regime $1/2 < \alpha < 1$, functions in the half-period cosine space $H_{\alpha,\gamma,1}^{\cos}$ need not possess a first-order weak derivative in $L_2([0,1])$, but rather exhibit fractional smoothness. To ensure that the characterized physical representation continuously connects to the integer case $\alpha=1$ as $\alpha \to 1^{-}$, we introduce a normalized, periodized Slobodeckij-type bilinear form. 

Let $\Pcal^{\cos} \coloneqq \operatorname{span}\left\{ \phi_k \mid k\in \NN_0\right\}$ denote the space of finite half-period cosine expansions. For $f,g\in \Pcal^{\cos}$, define 
\[
    \langle f,g \rangle_{\alpha,\gamma,1}^{\mathrm{sob}} \coloneqq \int_0^1 f(x) \rd x \int_0^1 g(x) \rd x + \frac{1}{\gamma^2} \Bcal_{\alpha}(f, g),
\]
where the Slobodeckij-type bilinear form $\Bcal_{\alpha}(f, g)$ is scaled by a specific regularizing factor that vanishes at $\alpha=1$:
\begin{align*}
    \Bcal_{\alpha}(f, g) & \coloneqq \frac{(1-\alpha)\pi^3}{8} \\
    & \quad \times \int_0^1 \int_0^1 \left( \frac{(f(x)-f(y))(g(x)-g(y))}{\left|\sin\left(\frac{\pi(x-y)}{2}\right)\right|^{2\alpha+1}} + \frac{(f(x)-f(y))(g(x)-g(y))}{\left|\sin\left(\frac{\pi(x+y)}{2}\right)\right|^{2\alpha+1}} \right)\rd x \rd y.
\end{align*}
We define $H_{\alpha,\gamma,1}^{\sob}$ as the completion of $\Pcal^{\cos}$ with respect to the norm induced by this inner product.

The normalization factor $(1-\alpha)$ is motivated by the Bourgain--Brezis--Mironescu (BBM) principle adapted to the periodic setting \cite{BBM01,BBM02,MS02}: as $\alpha \to 1^{-}$, it compensates for the
increasing singularity of the kernel. For sufficiently regular fixed functions, the resulting bilinear form converges to the classical Dirichlet form $\int_0^1 f'(x)g'(x) \rd x$. Accordingly, the normalization is chosen so that the fractional Sobolev-type norm connects continuously to the first-order unanchored Sobolev norm at $\alpha=1$. This limiting behavior for the induced norms will be proved in Theorem~\ref{thm:bbm_limiting}.

The following theorem identifies the completion introduced above with the corresponding half-period cosine space for $1/2<\alpha<1$.

\begin{theorem}\label{thm:equivalence1}
    For $1/2 < \alpha < 1$, the identity map on $\Pcal^{\cos}$ extends uniquely to an isomorphism between $H_{\alpha,\gamma,1}^{\cos}$ and $H_{\alpha,\gamma,1}^{\mathrm{sob}}$. Thus, after this canonical identification, the two spaces coincide as sets, and their norms are equivalent with constants independent of $\gamma$ and $f$.
\end{theorem}
\noindent The proof of Theorem~\ref{thm:equivalence1} will be provided in Subsection~\ref{subsec:proof}.

For the multivariate case, the corresponding weighted fractional Sobolev space, which is norm-equivalent to the weighted half-period cosine space $H_{\alpha,\bsgamma,d}^{\cos}$, is defined via the Hilbert tensor product structure as
\[ H_{\alpha,\bsgamma,d}^{\mathrm{sob}} \coloneqq \bigotimes_{j=1}^{d} H_{\alpha,\gamma_j,1}^{\mathrm{sob}}. \]

\begin{remark}
For $1/2<\alpha<1$, the bilinear form introduced above is equivalent to the classical Sobolev--Slobodeckij seminorm. Indeed, by using
\[ \left| \sin\left( \frac{\pi(x-y)}{2}\right) \right| \le \frac{\pi}{2}|x-y|, \]
as well as 
\[    \left|\sin\left(\frac{\pi(x+y)}{2}\right)\right| \ge \left|\sin\left(\frac{\pi(x-y)}{2}\right)\right|  \ge |x-y|, \]
for any $x,y\in [0,1]$, we obtain
\[
    \frac{(1-\alpha)\pi^3}{8}
    \left(\frac{2}{\pi}\right)^{2\alpha+1}
    |f|_{W^{\alpha,2}([0,1])}^2
    \leq
    \Bcal_\alpha(f,f)
    \leq
    \frac{(1-\alpha)\pi^3}{4}
    |f|_{W^{\alpha,2}([0,1])}^2,
\]
where 
\[ |f|_{W^{\alpha,2}([0,1])}^2 = \int_0^1 \int_0^1 \frac{|f(x)-f(y)|^2}{|x-y|^{2\alpha+1}}\rd x \rd y.\]
Together with the standard fractional Poincar\'{e} inequality and the density of finite cosine expansions in $W^{\alpha,2}([0,1])$, these estimates imply that, for every fixed $\gamma>0$, $H_{\alpha,\gamma,1}^{\sob}$ coincides as a set with the classical fractional Sobolev space $W^{\alpha,2}([0,1])$. The purpose of the periodized form $\Bcal_\alpha$ is therefore not to introduce a new underlying function space, but to provide an equivalent norm adapted to the half-period cosine basis. This representation also motivates the natural higher-order construction for $1<\alpha<2$ introduced in the next subsection. We finally note that related embedding results between several classical fractional smoothness spaces and spaces of Riemann--Liouville type can be found in \cite{GDMS26,MGxx,ST25}.
\end{remark}

%%%%%%%%%%%%%%%%%%%%%%%%%%%%%%%%%%%%%%%%%%%%%%%%%%
\subsection{The case of $1<\alpha<2$}
For the higher fractional range $1 < \alpha < 2$, functions in the half-period cosine space $H_{\alpha,\gamma,1}^{\cos}$ possess a first-order weak derivative $f' \in L_2([0,1])$, while $f'$ itself exhibits fractional smoothness of order $\alpha - 1 \in (0,1)$. In addition to the standard construction of higher-order fractional Sobolev spaces (see, e.g., \cite[Chapter~5]{M85}), we explicitly include the first-order derivative term in the inner product and penalize the fractional remainder of $f'$. 

Recall that $\Pcal^{\cos} \coloneqq \operatorname{span}\left\{ \phi_k \mid k\in \NN_0\right\}$ denotes the space of finite half-period cosine expansions. For $f,g\in \Pcal^{\cos}$, define
\[
    \langle f,g \rangle_{\alpha,\gamma,1}^{\mathrm{sob}} \coloneqq \int_0^1 f(x) \rd x \int_0^1 g(x) \rd x + \frac{1}{\gamma^2} \left( \int_0^1 f'(x)g'(x) \rd x + \Bcal_{\alpha}(f', g') \right),
\]
where the fractional bilinear form for the derivatives is modulated by vanishing factors at both endpoints:
\begin{align*}
    \Bcal_{\alpha}(f', g') & \coloneqq \frac{(\alpha-1)(2-\alpha)\pi^3}{8} \\
    & \quad \times \int_0^1 \int_0^1 \left( \frac{(f'(x)-f'(y))(g'(x)-g'(y))}{\left|\sin\left(\frac{\pi(x-y)}{2}\right)\right|^{2\alpha-1}} + \frac{(f'(x)+f'(y))(g'(x)+g'(y))}{\left|\sin\left(\frac{\pi(x+y)}{2}\right)\right|^{2\alpha-1}} \right) \rd x \rd y.
\end{align*}
Note that the positive sign in the numerator of the second term naturally arises from the odd-reflection property of the derivative $f'$, contrasting with the even-reflection applied to the function itself in the lower smoothness case $1/2 < \alpha < 1$. We define $H_{\alpha,\gamma,1}^{\sob}$ as the completion of $\Pcal^{\cos}$ with respect to the norm induced by this inner product.

As in the case with $1/2<\alpha<1$, the normalization factor $(\alpha - 1)(2 - \alpha)$, motivated by the BBM principle, is chosen to provide the appropriate limiting behavior at both integer endpoints. More precisely, for sufficiently regular fixed functions, one expects the following behavior:
\begin{itemize}
        \item As $\alpha \to 1^{+}$, the contribution of $\Bcal_\alpha(f',g')$ vanishes, and the inner product approaches that of $H_{1,\gamma,1}^{\mathrm{sob}}$.
        \item As $\alpha \to 2^{-}$, the bilinear form $\Bcal_\alpha(f',g')$ approaches $\int_0^1 f''(x)g''(x) \rd x$. Moreover, for smoothness $\alpha>3/2$, the boundary conditions $f'(0)=f'(1)=0$ arise naturally from the cosine representation.
\end{itemize}
We will prove these limiting behaviors for the induced norms later in Theorem~\ref{thm:bbm_limiting}.

The following theorem identifies the completion introduced above with the corresponding half-period cosine space for $1<\alpha<2$.

\begin{theorem}\label{thm:equivalence2}
    For $1 < \alpha < 2$, the identity map on $\Pcal^{\cos}$ extends uniquely to an isomorphism between $H_{\alpha,\gamma,1}^{\cos}$ and $H_{\alpha,\gamma,1}^{\mathrm{sob}}$. Thus, after this canonical identification, the two spaces coincide as sets, and their norms are equivalent with constants independent of $\gamma$ and $f$.
\end{theorem}
\noindent The proof of Theorem~\ref{thm:equivalence2} will be provided in Subsection~\ref{subsec:proof}.

Similarly to the lower smoothness case, for the multivariate case, we define
\[ H_{\alpha,\bsgamma,d}^{\mathrm{sob}} \coloneqq \bigotimes_{j=1}^{d} H_{\alpha,\gamma_j,1}^{\mathrm{sob}}, \]
where the tensor product is understood in the Hilbert-space sense. By
Theorem~\ref{thm:equivalence2} and the tensor product structure, this space is norm-equivalent to $H_{\alpha,\bsgamma,d}^{\cos}$.

%%%%%%%%%%%%%%%%%%%%%%%%%%%%%%%%%%%%%%%%%%%%%%%%%%
\subsection{Proofs of the norm-equivalence theorems}\label{subsec:proof}

Here, we provide the proofs of Theorems~\ref{thm:equivalence1} and \ref{thm:equivalence2}. A crucial step in evaluating the Slobodeckij-type bilinear forms $\Bcal_{\alpha}$ is to transform the integrals over the unit square $[0,1]^2$ into a single integral over $[-1,1]^2$ via symmetric extensions. For a function $f \in L_2([0,1])$, we define its even extension $f_E$ and odd extension $f_O$ onto $[-1,1]$ by
\[
    f_E(x) \coloneqq \begin{cases} f(x) & \text{if } x \in [0,1], \\ f(-x) & \text{if } x \in [-1,0), \end{cases} \quad \text{and} \quad
    f_O(x) \coloneqq \begin{cases} f(x) & \text{if } x \in [0,1], \\ -f(-x) & \text{if } x \in [-1,0). \end{cases}
\]

The following lemma establishes how these extensions simplify the target double integrals. The results follow from a straightforward domain partitioning of the square $[-1,1]^2$ into four quadrants and employing the respective symmetry properties of $f_E$ and $f_O$; hence, the detailed proof is omitted.

\begin{lemma} \label{lem:extension_trick}
    Let $\beta > 0$ and $f,g$ be measurable functions for which all the integrals below are absolutely convergent.
    \begin{enumerate}
        \item Under the even extension, it holds that
        \begin{align*}
            &\int_0^1 \int_0^1 \left(\frac{(f(x)-f(y))(g(x)-g(y))}{\left|\sin\left(\frac{\pi(x-y)}{2}\right)\right|^{\beta}} + \frac{(f(x)-f(y))(g(x)-g(y))}{\left|\sin\left(\frac{\pi(x+y)}{2}\right)\right|^{\beta}}\right) \rd x \rd y \\
            &= \frac{1}{2} \int_{-1}^1 \int_{-1}^1 \frac{(f_E(x)-f_E(y))(g_E(x)-g_E(y))}{\left|\sin\left(\frac{\pi(x-y)}{2}\right)\right|^{\beta}} \rd x \rd y.
        \end{align*}
        \item Under the odd extension, it holds that
        \begin{align*}
            &\int_0^1 \int_0^1 \left(\frac{(f(x)-f(y))(g(x)-g(y))}{\left|\sin\left(\frac{\pi(x-y)}{2}\right)\right|^{\beta}} + \frac{(f(x)+f(y))(g(x)+g(y))}{\left|\sin\left(\frac{\pi(x+y)}{2}\right)\right|^{\beta}} \right) \rd x \rd y \\
            &= \frac{1}{2} \int_{-1}^1 \int_{-1}^1 \frac{(f_O(x)-f_O(y))(g_O(x)-g_O(y))}{\left|\sin\left(\frac{\pi(x-y)}{2}\right)\right|^{\beta}} \rd x \rd y.
        \end{align*}
    \end{enumerate}
\end{lemma}

%%%%%%%%%%%%%%%%%%%%%%%%%%%%%%%%%%%%%%%%%%%%%%%%%%
\subsubsection{Proof of Theorem~\ref{thm:equivalence1} ($1/2 < \alpha < 1$)}
Recall that the set $\{\phi_k \mid k \in \NN_0\}$ forms a complete orthonormal basis for $L_2([0,1])$. Thus, any function $f \in H_{\alpha,\gamma,1}^{\cos}$ uniquely possesses the half-period cosine expansion
\[ f(x) = \sum_{k=0}^{\infty} \widetilde{f}(k) \phi_k(x). \]
The following lemma computes the action of the bilinear form $\Bcal_{\alpha}(\cdot, \cdot)$ on pairs of these basis functions for the lower fractional regime.
\begin{lemma}\label{lem:bilinear1}
    Let $\{\phi_k\mid k\in \NN_0\}$ be the set of half-period cosine functions, and let $\alpha\in (1/2,1)$. For $k,\ell\in \NN_0$, we have
    \begin{enumerate}
        \item If $k=0$ or $\ell=0$, $\Bcal_{\alpha}(\phi_k,\phi_\ell)=0$.
        \item If $k\neq \ell$, $\Bcal_{\alpha}(\phi_k,\phi_\ell)=0$.
        \item If $k=\ell>0$, there exist $0<c_{\alpha}<C_{\alpha}$, independent of $k$, such that
        \[ c_{\alpha}k^{2\alpha}\le \Bcal_{\alpha}(\phi_k,\phi_k) \le C_{\alpha}k^{2\alpha}. \]
    \end{enumerate}
\end{lemma}

\begin{proof}
    Since the $0$-th mode $\phi_0(x) = 1$, it is easy to verify that the first item of the claim, i.e., $\Bcal_{\alpha}(\phi_0, \cdot) = 0$, holds true.
    
    For $k \in \NN$, the even extension of $\phi_k(x) = \sqrt{2}\cos(\pi k x)$ onto $[-1,1]$ is simply $\phi_k(x)$ itself, which forms a part of the standard Fourier basis on $[-1,1]$ (see, e.g., \cite{IN08}). By the first item of Lemma~\ref{lem:extension_trick}, the bilinear form $\Bcal_{\alpha}(\phi_k, \phi_{\ell})$ simplifies to an integral over $[-1,1]^2$. For any $k,\ell\in \NN$ with $k\neq \ell$, we have
    \begin{align*}
        \Bcal_{\alpha}(\phi_k, \phi_{\ell}) & = \frac{(1-\alpha)\pi^3}{8}\int_{-1}^1 \int_{-1}^1 \frac{(\cos(\pi kx)-\cos(\pi ky))(\cos(\pi \ell x)-\cos(\pi \ell y))}{\left|\sin\left(\frac{\pi(x-y)}{2}\right)\right|^{2\alpha+1}} \rd x \rd y\\
        & = \frac{(1-\alpha)\pi^3}{2}\int_{-1}^1 \int_{-1}^1 \frac{\sin\left(\frac{\pi k(x-y)}{2}\right)\sin\left(\frac{\pi k(x+y)}{2}\right)\sin\left(\frac{\pi \ell (x-y)}{2}\right)\sin\left(\frac{\pi \ell(x+y)}{2}\right)}{\left|\sin\left(\frac{\pi(x-y)}{2}\right)\right|^{2\alpha+1}} \rd x \rd y\\
        & = (1-\alpha)\pi^3 \int_{-1}^1 \frac{\sin(\pi kv)\sin(\pi \ell v)}{|\sin(\pi v)|^{2\alpha+1}} \left( \int_{-1+|v|}^{1-|v|} \sin(\pi ku)\sin(\pi \ell u) \rd u \right) \rd v\\
        & = 2(-1)^{k+\ell}(1-\alpha)\pi^3 \int_{0}^1 \frac{\sin(\pi kv)\sin(\pi \ell v)}{|\sin(\pi v)|^{2\alpha+1}} \left( \frac{\sin(\pi(k+\ell)v)}{\pi(k+\ell)}-\frac{\sin(\pi(k-\ell)v)}{\pi(k-\ell)}  \right) \rd v,
    \end{align*}
    where we applied the change of variables $u=(x+y)/2$ and $v=(x-y)/2$ in the third equality. Since the integrand in the last expression is odd with respect to $v=1/2$, $\Bcal_{\alpha}(\phi_k, \phi_{\ell})=0$, proving the second item. 

    In the case of $k=\ell>0$, we have
    \begin{align*}
        \Bcal_{\alpha}(\phi_k, \phi_k) & = (1-\alpha)\pi^3 \int_{-1}^1 \frac{\sin^2(\pi kv)}{|\sin(\pi v)|^{2\alpha+1}} \left( \int_{-1+|v|}^{1-|v|} \sin^2(\pi ku) \rd u \right) \rd v\\
        & = 2(1-\alpha)\pi^3 \int_{0}^1 \frac{\sin^2(\pi kv)}{|\sin(\pi v)|^{2\alpha+1}} \left( 1 - v + \frac{\sin(2\pi kv)}{2\pi k}  \right) \rd v\\
        & = 2(1-\alpha)\pi^3 \int_{0}^{1/2} \frac{\sin^2(\pi kv)}{|\sin(\pi v)|^{2\alpha+1}} \rd v\\
        & = \frac{2(1-\alpha)\pi^3}{k}\int_{0}^{k/2}  \frac{\sin^2(\pi t)}{|\sin(\pi t/k)|^{2\alpha+1}} \rd t.
    \end{align*}
    As it holds that $\sin x\le x$ for any $x\ge 0$, $\Bcal_{\alpha}(\phi_k, \phi_k) $ is bounded below by
    \begin{align*}
        \Bcal_{\alpha}(\phi_k, \phi_k) & \ge \frac{2(1-\alpha)\pi^3}{k}\int_{0}^{k/2}  \frac{\sin^2(\pi t)}{|\pi t/k|^{2\alpha+1}} \rd t \\
        & = 2(1-\alpha)\pi^{2-2\alpha} k^{2\alpha}\int_{0}^{k/2} \frac{\sin^2(\pi t)}{t^{2\alpha+1}} \rd t \\
        & \ge 2(1-\alpha)\pi^{2-2\alpha} k^{2\alpha} \int_{0}^{1/2} \frac{\sin^2(\pi t)}{t^{2\alpha+1}} \rd t,
    \end{align*}
    where the convergence of the integral is guaranteed since the integrand behaves like $t^{1-2\alpha}$ near the origin, and $1-2\alpha\in(-1,0)$ for $\alpha\in(1/2,1)$.
    On the other hand, by $\sin x\ge 2x/\pi$ for any $0\le x\le \pi/2$, $\Bcal_{\alpha}(\phi_k, \phi_k) $ is bounded above by
    \begin{align*}
        \Bcal_{\alpha}(\phi_k, \phi_k) & \le \frac{2(1-\alpha)\pi^3}{k}\int_{0}^{k/2} \frac{\sin^2(\pi t)}{|2t/k|^{2\alpha+1}} \rd t\\
        & = \frac{(1-\alpha)\pi^3}{2^{2\alpha}} k^{2\alpha}\int_{0}^{k/2} \frac{\sin^2(\pi t)}{t^{2\alpha+1}} \rd t \\
        & \le \frac{(1-\alpha)\pi^3}{2^{2\alpha}} k^{2\alpha}\int_{0}^{\infty} \frac{\sin^2(\pi t)}{t^{2\alpha+1}} \rd t.
    \end{align*}
    Thus, we proved the last item of this lemma with constants
    \[ c_{\alpha}=2(1-\alpha)\pi^{2-2\alpha} \int_{0}^{1/2} \frac{\sin^2(\pi t)}{t^{2\alpha+1}} \rd t,\quad \text{and}\quad C_{\alpha}= \frac{(1-\alpha)\pi^3}{2^{2\alpha}}\int_{0}^{\infty} \frac{\sin^2(\pi t)}{t^{2\alpha+1}} \rd t. \]
    This completes the proof of the last item.
\end{proof}

We are now ready to prove Theorem~\ref{thm:equivalence1}.

\begin{proof}[Proof of Theorem~\ref{thm:equivalence1}]
    We first establish the norm equivalence on the dense subspace $\Pcal^{\cos}$. Let
    \[  p(x) = \sum_{k=0}^{N}\widetilde{p}(k)\phi_k(x) \in \Pcal^{\cos}. \]
    By the definition of the half-period cosine norm, we have
    \begin{align*}
        (\| p\|_{\alpha,\gamma,1}^{\cos})^2 = \sum_{k=0}^{N}|\widetilde{p}(k)|^2 r_{\alpha,\gamma}^2(k) = |\widetilde{p}(0)|^2 + \frac{1}{\gamma^2}\sum_{k=1}^{N}|\widetilde{p}(k)|^2 k^{2\alpha}.
    \end{align*}

    We next consider the Sobolev-type norm of $p$. Since the cosine expansion of $p$ is finite, the bilinearity of $\Bcal_\alpha$ and Lemma~\ref{lem:bilinear1} give
    \begin{align*}
        (\| p \|_{\alpha,\gamma,1}^{\sob})^2 & = \left(\int_0^1 p(x) \rd x \right)^2 + \frac{1}{\gamma^2} \Bcal_{\alpha}(p, p)\\
        & = |\widetilde{p}(0)|^2+\frac{1}{\gamma^2}\sum_{k,\ell=0}^{N}\widetilde{p}(k)\widetilde{p}(\ell)\Bcal_{\alpha}(\phi_k, \phi_{\ell}) \\
        & = |\widetilde{p}(0)|^2+\frac{1}{\gamma^2}\sum_{k=1}^{N}|\widetilde{p}(k)|^2\Bcal_{\alpha}(\phi_k, \phi_k),
    \end{align*}
    where, in the last equality, we used     $\Bcal_\alpha(\phi_0,\phi_\ell)=0$ for every $\ell\in\NN_0$ and     $\Bcal_\alpha(\phi_k,\phi_\ell)=0$ whenever $k\neq\ell$.

    Substituting the mode-wise bounds from Lemma~\ref{lem:bilinear1}, we obtain
    \[ |\widetilde{p}(0)|^2+\frac{c_{\alpha}}{\gamma^2}\sum_{k=1}^{N}|\widetilde{p}(k)|^2k^{2\alpha} \le (\| p \|_{\alpha,\gamma,1}^{\sob})^2\le |\widetilde{p}(0)|^2+\frac{C_{\alpha}}{\gamma^2}\sum_{k=1}^{N}|\widetilde{p}(k)|^2k^{2\alpha}. \]
    It follows that
    \[ \min\{1,c_{\alpha}\}(\| p\|_{\alpha,\gamma,1}^{\cos})^2\le (\| p \|_{\alpha,\gamma,1}^{\sob})^2\le \max\{1,C_{\alpha}\}(\| p\|_{\alpha,\gamma,1}^{\cos})^2. \]
    Hence, the cosine norm and the Sobolev-type norm are equivalent on $\Pcal^{\cos}$, with constants depending only on $\alpha$.

    The upper estimate shows that the identity map
    \[  \operatorname{id}\colon
        \left(\Pcal^{\cos},\|\cdot\|_{\alpha,\gamma,1}^{\cos}\right)
        \longrightarrow
        \left(\Pcal^{\cos},\|\cdot\|_{\alpha,\gamma,1}^{\sob}\right)
    \]
    is bounded, whereas the lower estimate shows that the identity map in the opposite direction is also bounded. Since $\Pcal^{\cos}$ is dense in $H_{\alpha,\gamma,1}^{\cos}$ and, by definition, in $H_{\alpha,\gamma,1}^{\sob}$, these two identity maps extend uniquely to bounded linear maps 
    \[
        T\colon H_{\alpha,\gamma,1}^{\cos} \longrightarrow H_{\alpha,\gamma,1}^{\sob} \quad \text{and}\quad S\colon H_{\alpha,\gamma,1}^{\sob} \longrightarrow H_{\alpha,\gamma,1}^{\cos},
    \]
    respectively. Since both compositions $S\circ T$ and $T\circ S$
    agree with the identity map on the dense subspace $\Pcal^{\cos}$,
    they are the identity maps on the corresponding completions.
    Therefore, $T$ is an isomorphism with inverse $S$. Consequently, after this canonical identification, the two spaces coincide as sets. Moreover, by continuity, the two-sided estimate extends from $\Pcal^{\cos}$ to the completions, and hence
    \[ \min\{1,c_{\alpha}\}(\| f\|_{\alpha,\gamma,1}^{\cos})^2\le (\| f \|_{\alpha,\gamma,1}^{\sob})^2\le \max\{1,C_{\alpha}\}(\| f\|_{\alpha,\gamma,1}^{\cos})^2 \]
    for every $f\in H_{\alpha,\gamma,1}^{\cos}=H_{\alpha,\gamma,1}^{\sob}$. The equivalence constants are independent of $\gamma$ and $f$.
\end{proof}

\begin{remark}\label{rem:tractability_transfer1}
    For the multivariate case, due to the tensor product construction and the mode-wise bounds established in the proof of Theorem~\ref{thm:equivalence1}, the multivariate norms satisfy the following explicit two-sided bounds:
    \[ \| f \|_{\alpha,c_{\alpha}^{-1/2}\bsgamma,d}^{\cos} \le \| f \|_{\alpha,\bsgamma,d}^{\mathrm{sob}} \le \| f \|_{\alpha,C_{\alpha}^{-1/2}\bsgamma,d}^{\cos}, \]
    where the positive constants $c_{\alpha}$ and $C_{\alpha}$ are given as in Lemma~\ref{lem:bilinear1}. Consequently, various tractability properties for multivariate integration and function approximation can be effortlessly transferred from the results for the weighted half-period cosine spaces to our newly introduced fractional Sobolev spaces by a simple scaling of the weight parameters.
\end{remark}

%%%%%%%%%%%%%%%%%%%%%%%%%%%%%%%%%%%%%%%%%%%%%%%%%%
\subsubsection{Proof of Theorem~\ref{thm:equivalence2} ($1 < \alpha < 2$)}

For the higher fractional smoothness range $\alpha\in (1,2)$, the bilinear form in the inner product of $H_{\alpha,\gamma,1}^{\mathrm{sob}}$ explicitly involves the first-order weak derivatives of the functions. As justified in Remark~\ref{rem:derivative_regularity} below, given the half-period cosine expansion of $f \in H_{\alpha,\gamma,1}^{\cos}$, its weak derivative $f'$ is given by the half-period sine expansion
\[ f'(x) = \sum_{k=1}^{\infty}\widetilde{f}(k) \phi'_k(x) = -\sqrt{2}\pi \sum_{k=1}^{\infty}k \widetilde{f}(k) \sin(\pi kx). \]
The following lemma computes the action of the bilinear form $\Bcal_{\alpha}(\cdot, \cdot)$ on pairs of these derivatives of the basis functions.

\begin{lemma}\label{lem:bilinear2}
    Let $\{\phi_k \mid k \in \NN_0\}$ be the set of half-period cosine basis functions, and let $\alpha \in (1,2)$. For any $k,\ell \in \NN$, we have:
    \begin{enumerate}
        \item If $k \neq \ell$, then $\Bcal_{\alpha}(\phi'_k,\phi'_\ell) = 0$.
        \item If $k = \ell$, there exist positive constants $c'_{\alpha}$ and $C'_{\alpha}$, independent of $k$, such that
        \[ c'_{\alpha}k^{2\alpha} \le \Bcal_{\alpha}(\phi'_k,\phi'_k) \le C_{\alpha}'k^{2\alpha}. \]
    \end{enumerate}
\end{lemma}

\begin{proof}
    For $k \in \NN$, the odd extension of $\phi'_k(x) = -\sqrt{2}\pi k \sin(\pi k x)$ onto $[-1,1]$ is simply $\phi'_k(x)$ itself. By applying the second item of Lemma~\ref{lem:extension_trick}, the bilinear form $\Bcal_{\alpha}(\phi'_k, \phi'_{\ell})$ simplifies to an integral over the square $[-1,1]^2$. For any $k,\ell \in \NN$ with $k \neq \ell$, we have
    \begin{align*}
        & \Bcal_{\alpha}(\phi'_k, \phi'_{\ell}) \\
        & = \frac{(\alpha-1)(2-\alpha)\pi^3}{16} \int_{-1}^1 \int_{-1}^1 \frac{(\phi'_k(x)-\phi'_k(y))(\phi'_{\ell}(x)-\phi'_{\ell}(y))}{\left|\sin\left(\frac{\pi(x-y)}{2}\right)\right|^{2\alpha-1}} \rd x \rd y\\
        & = \frac{(\alpha-1)(2-\alpha)\pi^5}{8} k\ell \int_{-1}^1 \int_{-1}^1 \frac{(\sin(\pi kx)-\sin(\pi ky))(\sin(\pi \ell x)-\sin(\pi \ell y))}{\left|\sin\left(\frac{\pi(x-y)}{2}\right)\right|^{2\alpha-1}} \rd x \rd y\\
        & = \frac{(\alpha-1)(2-\alpha)\pi^5}{2} k\ell \int_{-1}^1 \int_{-1}^1 \frac{\sin\left(\frac{\pi k(x-y)}{2}\right)\cos\left(\frac{\pi k(x+y)}{2}\right)\sin\left(\frac{\pi \ell(x-y)}{2}\right)\cos\left(\frac{\pi \ell(x+y)}{2}\right)}{\left|\sin\left(\frac{\pi(x-y)}{2}\right)\right|^{2\alpha-1}} \rd x \rd y\\
        & = (\alpha-1)(2-\alpha)\pi^5 k\ell \int_{-1}^1 \frac{\sin(\pi kv)\sin(\pi \ell v)}{|\sin(\pi v)|^{2\alpha-1}} \left( \int_{-1+|v|}^{1-|v|} \cos(\pi ku)\cos(\pi \ell u) \rd u \right) \rd v\\
        & = 2(-1)^{k+\ell+1}(\alpha-1)(2-\alpha)\pi^5 k\ell \int_{0}^1 \frac{\sin(\pi kv)\sin(\pi \ell v)}{|\sin(\pi v)|^{2\alpha-1}} \left( \frac{\sin(\pi(k+\ell)v)}{\pi(k+\ell)}+\frac{\sin(\pi(k-\ell)v)}{\pi(k-\ell)}  \right) \rd v,
    \end{align*}
    where we applied the change of variables $u=(x+y)/2$ and $v=(x-y)/2$ in the fourth equality. Since the integrand in the last expression is odd with respect to $v=1/2$, $\Bcal_{\alpha}(\phi'_k, \phi'_{\ell})=0$, which proves the first item.

    In the diagonal case where $k=\ell>0$, the same change of variables leads to
    \begin{align*}
        \Bcal_{\alpha}(\phi'_k, \phi'_k) & = (\alpha-1)(2-\alpha)\pi^5 k^2 \int_{-1}^1 \frac{\sin^2(\pi kv)}{|\sin(\pi v)|^{2\alpha-1}} \left( \int_{-1+|v|}^{1-|v|} \cos^2(\pi ku) \rd u \right) \rd v\\
        & = 2(\alpha-1)(2-\alpha)\pi^5 k^2 \int_{0}^1 \frac{\sin^2(\pi kv)}{|\sin(\pi v)|^{2\alpha-1}} \left( 1 - v - \frac{\sin(2\pi kv)}{2\pi k} \right) \rd v\\
        & = 2(\alpha-1)(2-\alpha)\pi^5 k^2 \int_{0}^{1/2} \frac{\sin^2(\pi kv)}{|\sin(\pi v)|^{2\alpha-1}} \rd v\\
        & = 2(\alpha-1)(2-\alpha)\pi^5 k \int_{0}^{k/2} \frac{\sin^2(\pi t)}{|\sin(\pi t/k)|^{2\alpha-1}} \rd t,
    \end{align*}
    where we have used the symmetry around $v=1/2$ and substituted $t = kv$.
    The rest of the argument is essentially the same as that of Lemma~\ref{lem:bilinear1}, and we obtain a lower bound
    \[ \Bcal_{\alpha}(\phi'_k, \phi'_k) \ge 2(\alpha-1)(2-\alpha)\pi^{6-2\alpha}k^{2\alpha}\int_{0}^{1/2} \frac{\sin^2(\pi t)}{t^{2\alpha-1}} \rd t, \]
    as well as an upper bound
    \[ \Bcal_{\alpha}(\phi'_k, \phi'_k) \le \frac{(\alpha-1)(2-\alpha)\pi^5}{2^{2\alpha-2}}k^{2\alpha}\int_{0}^{\infty} \frac{\sin^2(\pi t)}{t^{2\alpha-1}} \rd t.\]
    Here, the convergence of the integrals is guaranteed since the integrand behaves like $t^{3-2\alpha}$ near the origin, with $3-2\alpha\in(-1,1)$ for $\alpha\in(1,2)$, and the decay at infinity is of order $t^{1-2\alpha}$ with $2\alpha-1 > 1$. Thus, defining the $k$-independent constants
    \[ c'_{\alpha} \coloneqq 2(\alpha-1)(2-\alpha)\pi^{6-2\alpha}\int_{0}^{1/2} \frac{\sin^2(\pi t)}{t^{2\alpha-1}} \rd t, \quad \text{and} \quad C'_{\alpha} \coloneqq \frac{(\alpha-1)(2-\alpha)\pi^5}{2^{2\alpha-2}}\int_{0}^{\infty} \frac{\sin^2(\pi t)}{t^{2\alpha-1}} \rd t, \]
    we establish the second item. This completes the proof of the lemma.
\end{proof}

We are now ready to prove Theorem~\ref{thm:equivalence2}.

\begin{proof}[Proof of Theorem~\ref{thm:equivalence2}]
    The proof strategy is the same as that of Theorem~\ref{thm:equivalence1}. We first establish the norm equivalence on the dense subspace $\Pcal^{\cos}$. Let
    \[
        p(x)=\sum_{k=0}^{N}\widetilde{p}(k)\phi_k(x)
        \in\Pcal^{\cos}.
    \]
    By the definition of the half-period cosine norm, we have
    \begin{align*}
        (\| p\|_{\alpha,\gamma,1}^{\cos})^2 = \sum_{k=0}^{N}|\widetilde{p}(k)|^2 r_{\alpha,\gamma}^2(k) = |\widetilde{p}(0)|^2 + \frac{1}{\gamma^2}\sum_{k=1}^{N}|\widetilde{p}(k)|^2 k^{2\alpha}.
    \end{align*}

    We next consider the Sobolev-type norm of $p$. Since the cosine
    expansion of $p$ is finite, it can be differentiated term by term:
    \[
        p'(x) = \sum_{k=1}^{N}\widetilde{p}(k)\phi_k'(x) = -\sqrt{2}\pi \sum_{k=1}^{N} k\widetilde{p}(k)\sin(\pi kx).
    \]
    Hence, by the orthogonality of the sine functions,
    \[
        \int_0^1 (p'(x))^2\rd x = \pi^2 \sum_{k=1}^{N}k^2|\widetilde{p}(k)|^2.
    \]
    Moreover, the bilinearity of $\Bcal_\alpha$ and
    Lemma~\ref{lem:bilinear2} give
    \begin{align*}
        \Bcal_\alpha(p',p') = \sum_{k,\ell=1}^{N}\widetilde{p}(k)\widetilde{p}(\ell)\Bcal_\alpha(\phi_k',\phi_\ell') = \sum_{k=1}^{N}|\widetilde{p}(k)|^2\Bcal_\alpha(\phi_k',\phi_k'),
    \end{align*}
    where, in the last equality, we used $\Bcal_\alpha(\phi_k',\phi_\ell')=0$ whenever $k\neq \ell$.
    Consequently,
    \begin{align*}
        (\|p\|_{\alpha,\gamma,1}^{\sob})^2
        & =
        \left(\int_0^1p(x)\rd x\right)^2
        +
        \frac{1}{\gamma^2}
        \left(
            \int_0^1(p'(x))^2\rd x
            +
            \Bcal_\alpha(p',p')
        \right) \\
        &=
        |\widetilde{p}(0)|^2
        +
        \frac{1}{\gamma^2}
        \sum_{k=1}^{N}
        |\widetilde{p}(k)|^2
        \left(
            \pi^2k^2+
            \Bcal_\alpha(\phi_k',\phi_k')
        \right).
    \end{align*}

    By the mode-wise bounds established in
    Lemma~\ref{lem:bilinear2}, we have
    \[
        c'_\alpha k^{2\alpha}
        \leq
        \Bcal_\alpha(\phi_k',\phi_k')
        \leq
        C'_\alpha k^{2\alpha}.
    \]
    Since $k^2\leq k^{2\alpha}$ for every $k\in\NN$ and
    $\alpha>1$, it follows that
    \[
        c'_\alpha k^{2\alpha}
        \leq
        \pi^2k^2+\Bcal_\alpha(\phi_k',\phi_k')
        \leq
        (\pi^2+C'_\alpha)k^{2\alpha}.
    \]
    Defining
    \[
        C''_\alpha\coloneqq\pi^2+C'_\alpha,
    \]
    we therefore obtain
    \begin{align*}
        |\widetilde{p}(0)|^2
        +
        \frac{c'_\alpha}{\gamma^2}
        \sum_{k=1}^{N}
        |\widetilde{p}(k)|^2k^{2\alpha}
        \leq
        (\|p\|_{\alpha,\gamma,1}^{\sob})^2 
        \leq
        |\widetilde{p}(0)|^2
        +
        \frac{C''_\alpha}{\gamma^2}
        \sum_{k=1}^{N}
        |\widetilde{p}(k)|^2k^{2\alpha}.
    \end{align*}
    It follows that
    \[
        \min\{1,c'_\alpha\}
        (\|p\|_{\alpha,\gamma,1}^{\cos})^2
        \leq
        (\|p\|_{\alpha,\gamma,1}^{\sob})^2
        \leq
        \max\{1,C''_\alpha\}
        (\|p\|_{\alpha,\gamma,1}^{\cos})^2.
    \]
    Hence, the cosine norm and the Sobolev-type norm are equivalent on
    $\Pcal^{\cos}$, with constants depending only on $\alpha$.

    The remainder of the argument is identical to that in the proof of Theorem~\ref{thm:equivalence1}. Indeed, the two-sided estimate above shows that the identity maps on $\Pcal^{\cos}$ in both directions are bounded. Hence, they extend uniquely to bounded linear maps between $H_{\alpha,\gamma,1}^{\cos}$ and $H_{\alpha,\gamma,1}^{\sob}$. These extensions are inverse to each other, since their compositions agree with the identity map on the dense subspace $\Pcal^{\cos}$. Therefore, after the resulting canonical identification, the two spaces coincide as sets. Moreover, by continuity, the two-sided estimate extends to the completions, and hence
    \[
        \min\{1,c'_\alpha\}
        (\|f\|_{\alpha,\gamma,1}^{\cos})^2
        \leq
        (\|f\|_{\alpha,\gamma,1}^{\sob})^2
        \leq
        \max\{1,C''_\alpha\}
        (\|f\|_{\alpha,\gamma,1}^{\cos})^2
    \]
    for every $f\in H_{\alpha,\gamma,1}^{\cos}=H_{\alpha,\gamma,1}^{\sob}$. The equivalence constants are independent of $\gamma$ and $f$.
\end{proof}

\begin{remark}\label{rem:derivative_regularity}
    Let $f\in H_{\alpha,\gamma,1}^{\cos}$ with $1<\alpha<2$, and let
    $f_N(x) \coloneqq \sum_{k=0}^{N}\widetilde f(k)\phi_k(x)$. Since
    \begin{align*}
        \|f_N'-f_M'\|_{L_2([0,1])}^2 = \pi^2\sum_{k=M+1}^{N}        k^2|\widetilde f(k)|^2 \leq \pi^2 \sum_{k=M+1}^{N} k^{2\alpha}|\widetilde f(k)|^2,
    \end{align*}
    the sequence $(f_N')_{N\in\NN}$ is Cauchy in $L_2([0,1])$. Moreover, $f_N\to f$ in $L_2([0,1])$. Hence, there exists $g\in L_2([0,1])$ such that $f_N'\to g$ in $L_2([0,1])$, and the standard characterization of weak derivatives implies that $g$ is the first-order weak derivative of $f$. Consequently,
    \[
        f'(x)
        =
        \sum_{k=1}^{\infty}
        \widetilde f(k)\phi_k'(x)
    \]
    in the $L_2$-sense. In particular,
    \[
        \int_0^1(f'(x))^2\rd x
        =
        \pi^2
        \sum_{k=1}^{\infty}
        k^2|\widetilde f(k)|^2.
    \]

    If $\alpha>3/2$, then the Cauchy--Schwarz inequality gives
    \begin{align*}
        \sum_{k=1}^{\infty}
        k|\widetilde f(k)|
        &\leq
        \left(
            \sum_{k=1}^{\infty}
            k^{2\alpha}|\widetilde f(k)|^2
        \right)^{1/2}
        \left(
            \sum_{k=1}^{\infty}
            k^{2-2\alpha}
        \right)^{1/2} \\
        &\leq
        \gamma
        \|f\|_{\alpha,\gamma,1}^{\cos}
        \sqrt{\zeta(2\alpha-2)}
        <\infty.
    \end{align*}
    Therefore, the series representing $f'$ converges absolutely and
    uniformly on $[0,1]$. Since
    \[
        \phi_k'(0)=\phi_k'(1)=0
        \qquad\text{for every }k\in\NN,
    \]
    it follows that $f'$ has a continuous representative satisfying $f'(0)=f'(1)=0$.
\end{remark}

\begin{remark}\label{rem:tractability_transfer2}
    Similar to Remark~\ref{rem:tractability_transfer1}, the multivariate norms for $1 < \alpha < 2$ satisfy the following explicit two-sided bounds:
    \[ \| f \|_{\alpha,(c'_{\alpha})^{-1/2}\bsgamma,d}^{\cos} \le \| f \|_{\alpha,\bsgamma,d}^{\mathrm{sob}} \le \| f \|_{\alpha,(C''_{\alpha})^{-1/2}\bsgamma,d}^{\cos}, \]
    where the positive constants $c'_{\alpha}$ and $C''_{\alpha} = \pi^2 + C'_{\alpha}$ are given as above. Consequently, various tractability properties for multivariate integration and function approximation in the higher smoothness regime can be effortlessly transferred from the results for the weighted half-period cosine spaces to our newly introduced fractional Sobolev spaces by a simple scaling of the weight parameters.
\end{remark}

%%%%%%%%%%%%%%%%%%%%%%%%%%%%%%%%%%%%%%%%%%%%%%%%%%
\subsection{Limiting behavior and the BBM-type theorem}

Here, we formalize the limiting behavior of the fractional Sobolev norm as the smoothness parameter $\alpha$ approaches the integer boundaries $1$ and $2$. The resulting BBM-type formulas show how the fractional norms connect to the first-order unanchored Sobolev space at $\alpha=1$ and to a boundary-constrained second-order Sobolev space at $\alpha=2$.

For every $\alpha\in(1/2,2)$ and $k\in\NN$, we have
$k^{2\alpha}\leq k^4$, and therefore $\|f\|_{\alpha,\gamma,1}^{\cos} \leq \|f\|_{2,\gamma,1}^{\cos}$ for all $f\in H_{2,\gamma,1}^{\cos}$. Thus, $H_{2,\gamma,1}^{\cos}$ is continuously embedded in $H_{\alpha,\gamma,1}^{\cos}$ for every $\alpha \in (1/2,2)$.
Through the canonical identifications established in
Theorems~\ref{thm:equivalence1} and~\ref{thm:equivalence2}, we may
therefore regard every $f\in H_{2,\gamma,1}^{\cos}$ as an element of
$H_{\alpha,\gamma,1}^{\sob}$ for each
$\alpha\in(1/2,1)\cup(1,2)$.

\begin{theorem}\label{thm:bbm_limiting}
    Let $\gamma>0$ and $f\in H_{2,\gamma,1}^{\cos}$. Then the squared
    Sobolev-type norms have the following endpoint limits:
    \begin{enumerate}
        \item As $\alpha \to 1^-$ and $\alpha \to 1^+$, 
        \[
            \lim_{\alpha \to 1^-} (\| f \|_{\alpha,\gamma,1}^{\mathrm{sob}})^2 = \lim_{\alpha \to 1^+} (\| f \|_{\alpha,\gamma,1}^{\mathrm{sob}})^2 = \left( \int_0^1 f(x) \rd x \right)^2 + \frac{1}{\gamma^2} \int_0^1 (f'(x))^2 \rd x = (\| f \|_{1,\gamma,1}^{\mathrm{sob}})^2.
        \]
        \item As $\alpha\to2^-$,
        \[
            \lim_{\alpha\to2^-}(\|f\|_{\alpha,\gamma,1}^{\sob})^2 = \left(\int_0^1 f(x) \rd x\right)^2 + \frac{1}{\gamma^2}\left[ \int_0^1 (f'(x))^2 \rd x  + \int_0^1 (f''(x))^2 \rd x \right].
        \]
        Here, the limiting norm is understood on the boundary-constrained space $H_{2,\gamma,1}^{\sob(\mathrm{odd\text{-}bdry0})}$ rather than on the full second-order unanchored Sobolev space $H_{2,\gamma,1}^{\sob}$.
    \end{enumerate}
\end{theorem}

\begin{proof}
    By the integer-order characterization recalled in Subsection~\ref{subsec:known_equivalence} (see \cite[Lemma~1]{GSY19}), the spaces $H_{2,\gamma,1}^{\cos}$ and $H_{2,\gamma,1}^{\sob(\mathrm{odd\text{-}bdry0})}$ coincide as sets and have equivalent norms. Hence, the assumption $f\in H_{2,\gamma,1}^{\cos}$ implies that $f'(0)=f'(1)=0$ holds. This justifies the interpretation of the limiting expression in the second item as a modified norm on the boundary-constrained endpoint space.
    
    In what follows, we evaluate the limits by analyzing the behavior of the base integrals in Lemma~\ref{lem:bilinear1} and Lemma~\ref{lem:bilinear2} as $\alpha$ approaches the integer boundaries.

    Before considering the individual endpoint limits, we establish a
    uniform summable majorant that will justify all termwise limits
    below. Recall from Lemma~\ref{lem:bilinear1} that, for any $1/2<\alpha<1$ and $k\in \NN$, 
    \[
        \Bcal_\alpha(\phi_k,\phi_k)
        \leq
        C_\alpha k^{2\alpha}\quad \text{with}\quad C_\alpha
        =
        \frac{(1-\alpha)\pi^3}{2^{2\alpha}}
        \int_0^\infty
        \frac{\sin^2(\pi t)}{t^{2\alpha+1}}\rd t.
    \]
    For $\alpha\in[3/4,1)$, using
    $\sin^2(\pi t)\leq \pi^2t^2$ for $0<t\leq1$ and
    $\sin^2(\pi t)\leq1$ for $t\geq1$, we obtain
    \begin{align*}
        C_\alpha\leq
        \frac{(1-\alpha)\pi^3}{2^{2\alpha}}
        \left(
            \pi^2\int_0^1 t^{1-2\alpha}\,\rd t
            +
            \int_1^\infty t^{-2\alpha-1}\,\rd t
        \right) =
        \frac{\pi^3}{2^{2\alpha}}
        \left(
            \frac{\pi^2}{2}
            +
            \frac{1-\alpha}{2\alpha}
        \right).
    \end{align*}
    Hence, $C_\alpha$ is uniformly bounded for
    $\alpha\in[3/4,1)$. Similarly, for $1<\alpha<2$, Lemma~\ref{lem:bilinear2} gives
    \[
        \Bcal_\alpha(\phi_k',\phi_k')
        \leq
        C'_\alpha k^{2\alpha} \quad \text{with}\quad C'_\alpha
        =
        \frac{(\alpha-1)(2-\alpha)\pi^5}{2^{2\alpha-2}}
        \int_0^\infty
        \frac{\sin^2(\pi t)}{t^{2\alpha-1}}\rd t.
    \]
    The same elementary estimates yield
    \begin{align*}
        C'_\alpha \leq
        \frac{(\alpha-1)(2-\alpha)\pi^5}{2^{2\alpha-2}}
        \left(
            \pi^2\int_0^1 t^{3-2\alpha}\,\rd t
            +
            \int_1^\infty t^{1-2\alpha}\,\rd t
        \right) =
        \frac{\pi^5}{2^{2\alpha-2}}
        \left(
            \frac{\pi^2(\alpha-1)}{2}
            +
            \frac{2-\alpha}{2}
        \right).
    \end{align*}
    Thus, $C'_\alpha$ is uniformly bounded for
    $\alpha\in(1,2)$.

    Consequently, there exists a constant $M>0$, independent of $k$
    and $\alpha$ in the respective endpoint neighborhoods, such that
    \[
        \Bcal_\alpha(\phi_k,\phi_k)
        \leq Mk^4 \quad \text{for $\alpha\in[3/4,1)$,}\quad \text{and}\quad \Bcal_\alpha(\phi_k',\phi_k')
        \leq Mk^4 \quad \text{for $\alpha\in(1,2)$}.
    \]
    Since $f\in H_{2,\gamma,1}^{\cos}$, we have
    \[
        \sum_{k=1}^\infty
        |\widetilde f(k)|^2k^4
        <\infty.
    \]
    Therefore, the sequence
    \[
        \left(
            M|\widetilde f(k)|^2k^4
        \right)_{k\in\NN}
    \]
    provides a common summable majorant for all the series considered
    below. Hence, the corresponding termwise endpoint limits are
    justified by the dominated convergence theorem applied to the
    counting measure on $\NN$.
    
    To prove the first item from the lower smoothness regime, we apply the finite-sum identity established in the proof of Theorem~\ref{thm:equivalence1} to the cosine partial sums of $f$ and pass to the limit, obtaining
    \[
        (\| f \|_{\alpha,\gamma,1}^{\sob})^2 = |\widetilde{f}(0)|^2 + \frac{1}{\gamma^2}\sum_{k=1}^{\infty}|\widetilde{f}(k)|^2\Bcal_{\alpha}(\phi_k, \phi_k).
    \]
    As established in the proof of Lemma~\ref{lem:bilinear1}, the diagonal modes are explicitly given by
    \[
        \Bcal_{\alpha}(\phi_k, \phi_k) = \frac{2(1-\alpha)\pi^3}{k}\int_{0}^{k/2}  \frac{\sin^2(\pi t)}{|\sin(\pi t/k)|^{2\alpha+1}} \rd t.
    \]
    For the case $k=1$, it holds that
    \begin{align*}
        \Bcal_{\alpha}(\phi_1, \phi_1) & = 2(1-\alpha)\pi^2\int_{0}^{\pi/2}  \sin^{1-2\alpha}(t) \rd t\\
        & = 2(1-\alpha)\pi^2\cdot \frac{\sqrt{\pi}}{2}\frac{\Gamma(1-\alpha)}{\Gamma((3-2\alpha)/2)} = \frac{\pi^{5/2}\Gamma(2-\alpha)}{\Gamma((3-2\alpha)/2)},
    \end{align*}
    where $\Gamma$ denotes the Gamma function. Since $\Gamma(\cdot)$ is continuous on $(0,\infty)$, taking the limit $\alpha \to 1^-$ yields $\Gamma(2-\alpha) \to \Gamma(1) = 1$ and $\Gamma((3-2\alpha)/2) \to \Gamma(1/2) = \sqrt{\pi}$, which directly implies $\lim_{\alpha \to 1^-} \Bcal_{\alpha}(\phi_1, \phi_1) = \pi^2$. For the case $k>1$, we fix a small parameter $\delta \in (0,k/2)$ and split the integration domain into two parts:
    \begin{align*}
        \Bcal_{\alpha}(\phi_k, \phi_k) = \frac{2(1-\alpha)\pi^3}{k}\left(\int_{0}^{\delta}  \frac{\sin^2(\pi t)}{|\sin(\pi t/k)|^{2\alpha+1}} \rd t + \int_{\delta}^{k/2}  \frac{\sin^2(\pi t)}{|\sin(\pi t/k)|^{2\alpha+1}} \rd t\right).
    \end{align*}
    The second integral over the non-singular interval $[\delta, k/2]$ is easily bounded by
    \begin{align*}
        \int_{\delta}^{k/2}  \frac{\sin^2(\pi t)}{|\sin(\pi t/k)|^{2\alpha+1}} \rd t & \le \int_{\delta}^{k/2}  \frac{1}{|\sin(\pi \delta/k)|^{2\alpha+1}} \rd t\\
        & \le \frac{k}{2|\sin(\pi \delta/k)|^{2\alpha+1}} \to \frac{k}{2|\sin(\pi \delta/k)|^{3}}\quad \text{(as $\alpha\to 1^{-}$).}
    \end{align*}
    Since this integral remains bounded, the factor $(1-\alpha)$ ensures that
    \[ \lim_{\alpha\to 1^{-}}\frac{2(1-\alpha)\pi^3}{k} \int_{\delta}^{k/2}  \frac{\sin^2(\pi t)}{|\sin(\pi t/k)|^{2\alpha+1}} \rd t=0.\]
    For fixed $k$ and $\delta$, there exist $\alpha_0\in(1/2,1)$ and a constant $C>0$, independent of $\alpha\in[\alpha_0,1)$ and $t\in(0,\delta)$, such that 
    \[ \left| \frac{\sin^2(\pi t)}{|\sin(\pi t/k)|^{2\alpha+1}} - \pi^{1-2\alpha}k^{2\alpha+1}t^{1-2\alpha} \right| \le C t^{3-2\alpha}. \]
    With this uniform bound, we see that
    \begin{align*}
        & \lim_{\alpha\to 1^{-}}\frac{2(1-\alpha)\pi^3}{k}\left|\int_{0}^{\delta}  \frac{\sin^2(\pi t)}{|\sin(\pi t/k)|^{2\alpha+1}} \rd t-\int_{0}^{\delta}  \pi^{1-2\alpha}k^{2\alpha+1}t^{1-2\alpha} \rd t\right| \\
        & \le \lim_{\alpha\to 1^{-}}\frac{2(1-\alpha)\pi^3}{k}\int_{0}^{\delta}  \left|\frac{\sin^2(\pi t)}{|\sin(\pi t/k)|^{2\alpha+1}}- \pi^{1-2\alpha}k^{2\alpha+1}t^{1-2\alpha}\right| \rd t\\
        & \le \lim_{\alpha\to 1^{-}}\frac{2(1-\alpha)\pi^3}{k}C\int_{0}^{\delta} t^{3-2\alpha} \rd t=\lim_{\alpha\to 1^{-}}\frac{2(1-\alpha)\pi^3}{k}\cdot\frac{C\delta^{4-2\alpha}}{4-2\alpha}=0.
    \end{align*}
    This leads to
    \begin{align*}
        \lim_{\alpha\to 1^{-}}\Bcal_{\alpha}(\phi_k, \phi_k) & = \lim_{\alpha\to 1^{-}}\frac{2(1-\alpha)\pi^3}{k}\int_{0}^{\delta}  \pi^{1-2\alpha}k^{2\alpha+1}t^{1-2\alpha} \rd t\\
        & = \lim_{\alpha\to 1^{-}}\pi^{4-2\alpha}2(1-\alpha)k^{2\alpha}\int_{0}^{\delta}  t^{1-2\alpha} \rd t\\
        & = \lim_{\alpha\to 1^{-}}\pi^{4-2\alpha}k^{2\alpha}\delta^{2-2\alpha} = \pi^2 k^2.
    \end{align*}
    By the uniform domination established above, we may pass to the limit term by term in the series, and hence
    \begin{align*}
        \lim_{\alpha\to 1^{-}}(\| f \|_{\alpha,\gamma,1}^{\sob})^2 & = |\widetilde{f}(0)|^2 + \frac{\pi^2}{\gamma^2}\sum_{k=1}^{\infty}|\widetilde{f}(k)|^2 k^2\\
        & = \left( \int_0^1 f(x)\rd x\right)^2+\frac{1}{\gamma^2}\int_0^1 (f'(x))^2 \rd x = (\| f \|_{1,\gamma,1}^{\sob})^2,
    \end{align*}
    which completes the proof of the first item from the lower regime.

    Next, we move on to proving the first item from the higher smoothness regime. By applying the finite-sum identity established in the proof of Theorem~\ref{thm:equivalence2} to the cosine partial sums of $f$ and passing to the limit, we obtain, for every $1<\alpha<2$,
    \[
    (\|f\|_{\alpha,\gamma,1}^{\sob})^2 = |\widetilde f(0)|^2 + \frac{1}{\gamma^2} \sum_{k=1}^{\infty} |\widetilde f(k)|^2 \left( \pi^2k^2+ \Bcal_\alpha(\phi_k',\phi_k')\right).
    \]
    For $1 < \alpha < 2$, as detailed in Lemma~\ref{lem:bilinear2}, the bilinear form on the derivatives yields
    \begin{align*}
        \Bcal_{\alpha}(\phi'_k, \phi'_k) = 2(\alpha-1)(2-\alpha)\pi^5 k \int_{0}^{k/2} \frac{\sin^2(\pi t)}{|\sin(\pi t/k)|^{2\alpha-1}} \rd t.
    \end{align*}
    For the case $k = 1$, we have
    \begin{align*}
    \Bcal_{\alpha}(\phi_1',\phi_1') & = 2(\alpha-1)(2-\alpha)\pi^4\int_0^{\pi/2}\sin ^{3-2\alpha} (t) \rd t\\
    &= 2(\alpha-1)(2-\alpha)\pi^4\frac{\sqrt{\pi}}{2} \frac{\Gamma(2-\alpha)}{\Gamma((5-2\alpha)/2)} \\
    & = \pi^{9/2}(\alpha-1)\frac{\Gamma(3-\alpha)}{\Gamma((5-2\alpha)/2)} .
    \end{align*}
    As $\alpha \to 1^+$, noting that $\Gamma(3-\alpha)\to \Gamma(2)=1$ and $\Gamma((5-2\alpha)/2)\to \Gamma(3/2)=\sqrt{\pi}/2$, the factor $(\alpha-1)$ ensures that $\Bcal_{\alpha}(\phi_1',\phi_1')\to 0$.
    For the case $k > 1$, the integral can be bounded from above as
    \begin{align*}
    0 < \int_{0}^{k/2} \frac{\sin^2(\pi t)}{|\sin(\pi t/k)|^{2\alpha-1}} \rd t \le \int_{0}^{k/2} \frac{(\pi t)^2}{(2t/k)^{2\alpha-1}} \rd t = \frac{\pi^2 k^3}{2^4 (2-\alpha)},
    \end{align*}
    where the upper bound is due to the inequalities $\sin(\pi t) \le \pi t$ and $\sin(\pi t / k) \ge 2t/k$ for any $t\in [0,k/2]$. Thus, we readily obtain
    \begin{align*}
        0 \le \limsup_{\alpha \to 1^+}\Bcal_{\alpha}(\phi'_k, \phi'_k) \le \lim_{\alpha \to 1^+}\frac{(\alpha-1)\pi^7 k^4}{2^3} = 0.
    \end{align*}
    Therefore, by the uniform domination established above, we may pass to the limit term by term in the series, which gives
    \begin{align*}
        \lim_{\alpha\to 1^{+}}(\| f \|_{\alpha,\gamma,1}^{\mathrm{sob}})^2 & = \lim_{\alpha\to 1^{+}}\left[\left(\int_0^1 f(x) \rd x\right)^2 + \frac{1}{\gamma^2} \left( \int_0^1 (f'(x))^2 \rd x + \Bcal_{\alpha}(f', f')\right)\right]\\
        & = \lim_{\alpha\to 1^{+}}\left[\left(\int_0^1 f(x) \rd x\right)^2 + \frac{1}{\gamma^2} \left( \int_0^1 (f'(x))^2 \rd x + \sum_{k=1}^{\infty}|\widetilde{f}(k)|^2 \Bcal_{\alpha}(\phi'_k, \phi'_k)\right)\right]\\
        & = \left(\int_0^1 f(x) \rd x\right)^2 + \frac{1}{\gamma^2} \int_0^1 (f'(x))^2 \rd x = (\| f \|_{1,\gamma,1}^{\mathrm{sob}})^2,
    \end{align*}
    which completes the proof of the first item from the higher regime.

    Finally, we evaluate the limiting behavior as $\alpha \to 2^-$ to prove the second item. Recall that we have
    \begin{align*}
        \Bcal_{\alpha}(\phi'_k, \phi'_k) & = 2(\alpha-1)(2-\alpha)\pi^5 k \int_{0}^{k/2} \frac{\sin^2(\pi t)}{|\sin(\pi t/k)|^{2\alpha-1}} \rd t \\
        & = 2(\alpha-1)(2-\alpha)\pi^5 k \left(\int_{0}^{\delta} \frac{\sin^2(\pi t)}{|\sin(\pi t/k)|^{2\alpha-1}} \rd t+\int_{\delta}^{k/2} \frac{\sin^2(\pi t)}{|\sin(\pi t/k)|^{2\alpha-1}} \rd t\right),
    \end{align*}
    where the last equality is obtained by fixing a small parameter $\delta \in (0,k/2)$ and splitting the integration domain into two parts. The second integral over the non-singular interval $[\delta, k/2]$ is easily bounded by
    \begin{align*}
        \int_{\delta}^{k/2} \frac{\sin^2(\pi t)}{|\sin(\pi t/k)|^{2\alpha-1}} \rd t & \le \int_{\delta}^{k/2} \frac{1}{|\sin(\pi \delta/k)|^{2\alpha-1}} \rd t \\
        & \le \frac{k}{2|\sin(\pi \delta/k)|^{2\alpha-1}} \to \frac{k}{2|\sin(\pi \delta/k)|^{3}} \quad \text{(as $\alpha\to 2^{-}$).}
    \end{align*}
    Since this integral remains bounded, the factor $(2-\alpha)$ ensures that
    \[ \lim_{\alpha\to 2^{-}}2(\alpha-1)(2-\alpha)\pi^5 k \int_{\delta}^{k/2}  \frac{\sin^2(\pi t)}{|\sin(\pi t/k)|^{2\alpha-1}} \rd t=0.\]
    For fixed $k$ and $\delta$, there exist $\alpha_0\in(1,2)$ and a constant $C>0$, independent of $\alpha\in[\alpha_0,2)$ and $t\in(0,\delta)$, such that 
    \[ \left| \frac{\sin^2(\pi t)}{|\sin(\pi t/k)|^{2\alpha-1}} - \pi^{3-2\alpha} k^{2\alpha-1} t^{3-2\alpha} \right| \le C t^{5-2\alpha}. \]
    As in the proof of the first item from the lower regime, we obtain
    \begin{align*}
        \lim_{\alpha\to 2^-} \Bcal_{\alpha}(\phi'_k, \phi'_k) & = \lim_{\alpha\to 2^-} 2(\alpha-1)(2-\alpha)\pi^5 k \int_0^{\delta} \pi^{3-2\alpha} k^{2\alpha-1} t^{3-2\alpha} \rd t \\
        & = \lim_{\alpha\to 2^-} 2(\alpha-1)\pi^{8-2\alpha} k^{2\alpha} \frac{2-\alpha}{4-2\alpha} \delta^{4-2\alpha} = \pi^4 k^4.
    \end{align*}
    By the same uniform domination, we may again pass to the limit term by term in the series. Therefore, the squared fractional Sobolev-type norm converges to
    \begin{align*}
        \lim_{\alpha\to 2^{-}}(\| f \|_{\alpha,\gamma,1}^{\mathrm{sob}})^2 = \left(\int_0^1 f(x) \rd x\right)^2 + \frac{1}{\gamma^2}\left( \int_0^1 (f'(x))^2 \rd x + \pi^4 \sum_{k=1}^{\infty}|\widetilde{f}(k)|^2 k^4 \right).
    \end{align*}
    Given the half-period cosine expansion of $f \in H_{2,\gamma,1}^{\mathrm{sob}(\mathrm{odd-bdry0})}$, its second-order weak derivative is expanded as $f''(x) = -\pi^2 \sum_{k=1}^{\infty} k^2 \widetilde{f}(k) \phi_k(x)$. Indeed, if $f_N(x)=\sum_{k=0}^{N}\widetilde{f}(k)\phi_k(x)$, then 
    \[ \|f''_N-f''_M\|_{L_2([0,1])}^2 = \pi^4 \sum_{k=M+1}^{N}k^4 |\widetilde{f}(k)|^2 \longrightarrow 0.\]
    Since $f_N\to f$ in $L_2([0,1])$, the $L_2$-limit of $f''_N$ is the second weak derivative of $f$. Hence, $f''(x) = -\pi^2 \sum_{k=1}^{\infty} k^2 \widetilde{f}(k) \phi_k(x)$ in the $L_2$-sense. Since the basis functions $\phi_k(x) = \sqrt{2}\cos(\pi k x)$ are orthonormal, Parseval's identity guarantees that the $L_2$-norm satisfies
    \[ \int_0^1 (f''(x))^2 \rd x = \pi^4 \sum_{k=1}^{\infty} k^4 |\widetilde{f}(k)|^2. \]
    Substituting this relationship back into the norm expression leads to
    \[ \lim_{\alpha\to 2^{-}}(\| f \|_{\alpha,\gamma,1}^{\mathrm{sob}})^2 = \left(\int_0^1 f(x) \rd x\right)^2 + \frac{1}{\gamma^2}\left( \int_0^1 (f'(x))^2 \rd x + \int_0^1 (f''(x))^2 \rd x \right), \]
    which completes the proof of the second item.
\end{proof}

\begin{remark}\label{rem:absolute_convergence}
    Let $1/2<\alpha<2$, $d\in\NN$, and let $\bsgamma=(\gamma_1,\ldots,\gamma_d)$ be a collection of positive weights. We see that every function $f\in H_{\alpha,\bsgamma,d}^{\sob}$ has an absolutely and uniformly convergent half-period cosine expansion. Indeed, by the norm equivalence established above, we may regard
    $f$ as an element of $H_{\alpha,\bsgamma,d}^{\cos}$. Since $\|\phi_{\bsk}\|_{L_\infty([0,1]^d)}=2^{|\bsk|_0/2}$, the Cauchy--Schwarz inequality yields
    \begin{align*}
        \sum_{\bsk\in\NN_0^d}
        |\widetilde f(\bsk)|
        \|\phi_{\bsk}\|_{L_\infty}
        &\leq
        \left(
            \sum_{\bsk\in\NN_0^d}
            |\widetilde f(\bsk)|^2
            r_{\alpha,\bsgamma}^2(\bsk)
        \right)^{1/2}
        \left(
            \sum_{\bsk\in\NN_0^d}
            \frac{2^{|\bsk|_0}}
                 {r_{\alpha,\bsgamma}^2(\bsk)}
        \right)^{1/2} \\
        &=
        \|f\|_{\alpha,\bsgamma,d}^{\cos}
        \prod_{j=1}^d
        \left(
            1+2\gamma_j^2\zeta(2\alpha)
        \right)^{1/2}
        <\infty,
    \end{align*}
    where the finiteness follows from $\alpha>1/2$. Hence, the
    half-period cosine expansion of $f$ converges absolutely and
    uniformly on $[0,1]^d$.
\end{remark}

In summary, we have established a family of weighted fractional and integer-order Sobolev spaces that continuously cover the range $1/2 < \alpha \le 2$, with the corresponding endpoint connections at
$\alpha=1$ and $\alpha=2$. Due to the norm equivalence to the half-period cosine space, results for multivariate problems in the half-period cosine setting can be effortlessly transferred to our novel Sobolev spaces. In what follows, we exploit this equivalence to study numerical integration and function approximation in Sections~\ref{sec:integration} and \ref{sec:approximation}, respectively.

%%%%%%%%%%%%%%%%%%%%%%%%%%%%%%%%%%%%%%%%%%%%%%%%%%
%%%%%%%%%%%%%%%%%%%%%%%%%%%%%%%%%%%%%%%%%%%%%%%%%%
\section{Numerical integration}\label{sec:integration}

In this section, we study multivariate numerical integration over $[0,1]^d$. We consider the problem of approximating the integral
\[ I(f) \coloneqq \int_{[0,1]^d}f(\bsx)\rd \bsx \]
 of a function $f\in H_{\alpha,\bsgamma,d}^{\sob}$. For a normed space $H$ of functions on $[0,1]^d$ and a deterministic integration algorithm $A\colon H\to\RR$, we define its worst-case error by
\[
    e^{\wor}(H;A) \coloneqq \sup_{\substack{f\in H\\ \|f\|_H\leq 1}} \left|I(f)-A(f)\right|.
\]

Due to Theorem~\ref{thm:equivalence1} and Remark~\ref{rem:tractability_transfer1}, the integer-order equivalence
recalled in Subsection~\ref{subsec:known_equivalence}, and Theorem~\ref{thm:equivalence2} and Remark~\ref{rem:tractability_transfer2}, existing error bounds, and consequently the associated tractability results, for weighted half-period cosine spaces can be transferred to our fractional Sobolev spaces through an appropriate coordinate-wise rescaling of the weights.

More precisely, let 
\[ \underline{\kappa}_{\alpha}\coloneqq \begin{cases}
c_{\alpha}, & \text{$1/2<\alpha<1$,}\\
\pi^2, & \text{$\alpha=1$,}\\
c'_{\alpha}, & \text{$1<\alpha<2$,}\\
\end{cases}\quad \text{and}\quad \overline{\kappa}_{\alpha}\coloneqq \begin{cases}
C_{\alpha}, & \text{$1/2<\alpha<1$,}\\
\pi^2, & \text{$\alpha=1$,}\\
C''_{\alpha}, & \text{$1<\alpha<2$.}\\
\end{cases} \]
where $c_\alpha, C_\alpha$ (for $1/2 < \alpha < 1$) and $c'_\alpha, C''_\alpha$ (for $1 < \alpha < 2$) are the norm-equivalence constants derived in Section~\ref{sec:equivalence}. The corresponding
unit balls satisfy
\[
    B\left( H_{\alpha,\overline{\kappa}_{\alpha}^{-1/2}\bsgamma,d}^{\cos} \right) \subseteq B\left(H_{\alpha,\bsgamma,d}^{\sob}\right) \subseteq B\left(H_{\alpha,\underline{\kappa}_{\alpha}^{-1/2}\bsgamma,d}^{\cos}\right),
\] 
where $B(H)\coloneqq\{f\in H\mid \|f\|_H\leq1\}$ denotes the closed
unit ball of $H$.
Consequently, for any deterministic algorithm $A$, the corresponding worst-case errors satisfy
\[ e^{\mathrm{wor}}\left(H_{\alpha,\overline{\kappa}_{\alpha}^{-1/2}\bsgamma,d}^{\cos}; A\right)\le e^{\mathrm{wor}}\left(H_{\alpha,\bsgamma,d}^{\mathrm{sob}}; A\right) \le e^{\mathrm{wor}}\left(H_{\alpha,\underline{\kappa}_{\alpha}^{-1/2}\bsgamma,d}^{\cos}; A\right). \]
Thus, an upper bound for $e^{\mathrm{wor}}(H_{\alpha,\bsgamma,d}^{\mathrm{sob}}; A)$ is obtained by applying the corresponding cosine-space result with the rescaled weights $\Gamma_{j,\alpha}=\underline{\kappa}_{\alpha}^{-1/2}\gamma_j$ for all $j=1,\ldots,d$. An analogous sandwich bound for randomized errors will be used in Subsection~\ref{subsec:randomized_integration}, after the randomized error criterion has been introduced.

In what follows, we leverage this automated transfer mechanism to establish error guarantees for both deterministic and randomized integration algorithms. In particular, we focus on tent-transformed rank-1 lattice rules from \cite{DNP14} in the deterministic setting, and their randomized variants studied in \cite{DGS22,GK26}.

%%%%%%%%%%%%%%%%%%%%%%%%%%%%%%%%%%%%%%%%%%%%%%%%%%
\subsection{Deterministic tent-transformed lattice rules}
As studied in \cite{DNP14}, we approximate $I(f)$ by an equally weighted quadrature rule
\[ I(f;P^{\varphi}_{n,\bsz})=\frac{1}{n}\sum_{i=0}^{n-1}f(\bsx_i), \]
where $P^{\varphi}_{n,\bsz}=\{\bsx_0,\ldots,\bsx_{n-1}\}\subset [0,1]^d$ denotes the tent-transformed rank-1 lattice point set. Note that the use of the tent transformation in the context of lattice rules was originally investigated by Hickernell in \cite{H02}.

Let the smoothness parameter $\alpha>1/2$ and the weight parameters $\bsgamma$ be given. In \cite[Corollary~1]{DNP14}, Dick et al.\ proved that the worst-case error of the tent-transformed lattice rule, with $n$ being a prime number and the generating vector $\bsz \in \{1,\ldots,n-1\}^d$ constructed via the fast component-by-component (CBC) algorithm, is bounded by
\begin{align*}
    e^{\mathrm{wor}}(H_{\alpha,\bsgamma,d}^{\cos}; P^{\varphi}_{n,\bsz}) & \coloneqq \sup_{\substack{f\in H_{\alpha,\bsgamma,d}^{\cos}\\ \|f\|_{\alpha,\bsgamma,d}^{\cos}\le 1}}\left| I(f)-I(f;P^{\varphi}_{n,\bsz})\right| \\
    & \: \le \left[ \frac{1}{n-1}\left(-1+\prod_{j=1}^{d}(1+2\gamma_j^{1/\lambda}\zeta(\alpha / \lambda))\right)\right]^{\lambda},
\end{align*}
for any $1/2\le \lambda<\alpha$, where we have applied a slight modification to adjust the notations. Here, the essential result is that the worst-case error of the tent-transformed lattice rule in $H_{\alpha,\bsgamma,d}^{\cos}$ is bounded above by that of the (non-tent-transformed) lattice rule in the weighted Korobov space with the same smoothness parameter $\alpha$ and weights $\bsgamma$, see \cite[Theorem~7.40]{DKP22}. 

Thus, the fast CBC algorithm employed here is the same as the one in \cite[Section~3.4]{DKP22} with rescaled weights, and is described in Algorithm~\ref{alg:cbc_lattice}. Note that, as proposed in \cite{NC06}, by exploiting the circulant matrix structure of the kernel evaluations, this greedy algorithm can be performed via the fast Fourier transform in $\mathcal{O}(d n \log n)$ operations.

\begin{algorithm}[H]
\caption{CBC construction for deterministic rank-1 lattice rules}
\label{alg:cbc_lattice}
\begin{algorithmic}[1]
\REQUIRE A prime number $n$, dimension $d$, smoothness $1/2<\alpha<2$, and rescaled weights $\bsGamma_{\alpha}\coloneqq \underline{\kappa}_{\alpha}^{-1/2}\bsgamma$.
\STATE Set $z_1 = 1$.
\FOR{$s = 2$ \TO $d$}
    \STATE Find $z_s \in \{1, \ldots, n-1\}$ that minimizes the worst-case criterion:
    \[
        E^2_{n,s}(z_s \mid (z_1, \ldots, z_{s-1})) \coloneqq -1 + \frac{1}{n} \sum_{i=0}^{n-1} \prod_{j=1}^s \left( 1 + \Gamma_{j,\alpha}^2 \omega_{2\alpha}\left( \left\{ \frac{i z_j}{n} \right\} \right) \right),
    \]
    where $\{ \cdot \}$ denotes the fractional part, and $\omega_{2\alpha}(x) \coloneqq \sum_{k \in \ZZ \setminus \{0\}} |k|^{-2\alpha} e^{2\pi \mi k x}$.
\ENDFOR
\RETURN $\bsz = (z_1, \ldots, z_d)$
\end{algorithmic}
\end{algorithm}

Due to the norm equivalence (Remarks~\ref{rem:tractability_transfer1} and \ref{rem:tractability_transfer2}), this deterministic error guarantee immediately yields convergence rates arbitrarily close to the optimal rate in our fractional Sobolev spaces.

\begin{theorem}\label{thm:integration_deterministic}
    Let $1/2<\alpha<2$. For a prime $n$ and the generating vector $\bsz$ constructed by the fast CBC algorithm as above, the worst-case error in the fractional Sobolev space $H_{\alpha,\bsgamma,d}^{\mathrm{sob}}$ satisfies
    \begin{align*}
        e^{\mathrm{wor}}(H_{\alpha,\bsgamma,d}^{\mathrm{sob}}; P^{\varphi}_{n,\bsz}) \le \left[ \frac{V_{d}(\alpha,\lambda,\bsgamma)-1}{n-1}\right]^{\lambda} \quad \text{with}\quad V_{d}(\alpha,\lambda,\bsgamma) \coloneqq \prod_{j=1}^{d}\left(1+2\Gamma_{j,\alpha}^{1/\lambda}\zeta(\alpha / \lambda)\right),
    \end{align*}
    for any $1/2\le \lambda<\alpha$, where the rescaled weights are given as $\Gamma_{j,\alpha}=\underline{\kappa}_{\alpha}^{-1/2}\gamma_j$.
\end{theorem}

\begin{remark}\label{rem:tractability_deterministic}
    From Theorem~\ref{thm:integration_deterministic}, we can easily derive tractability results for our novel fractional Sobolev spaces. If the weights $\bsgamma = (\gamma_j)_{j \ge 1}$ satisfy the summability condition
    \[ \sum_{j=1}^{\infty} \gamma_j^{1/\lambda} < \infty \]
    for some $1/2 \le \lambda < \alpha$, then the upper bound in Theorem~\ref{thm:integration_deterministic} can be bounded independently of the dimension $d$. This immediately implies that numerical integration in $H_{\alpha,\bsgamma,d}^{\mathrm{sob}}$ achieves \emph{strong polynomial tractability} with a convergence rate of $\mathcal{O}(n^{-\lambda})$. Thus, to achieve convergence rates arbitrarily close to the optimal rate while retaining strong polynomial tractability, a sufficient condition is 
    \[ \sum_{j=1}^{\infty} \gamma_j^{1/\alpha} < \infty. \]    
\end{remark}

%%%%%%%%%%%%%%%%%%%%%%%%%%%%%%%%%%%%%%%%%%%%%%%%%%
\subsection{Randomized tent-transformed lattice rules}\label{subsec:randomized_integration}

In the randomized setting, instead of fixing a point set $P_{n,\bsz}^{\varphi}$ deterministically, we allow for randomness in the choice of the quadrature rule. More generally, a randomized algorithm is defined by a pair consisting of a probability space $(\Omega, \Sigma,\mu)$ and a family of deterministic quadrature rules $Q=(Q^{\omega})_{\omega\in \Omega}$, such that, for any fixed $\omega \in \Omega$, $Q^{\omega}(f)$ refers to a deterministic quadrature of $f$. As an error criterion, we employ the randomized error defined for a normed space $H$ of integrands by
\[ e^{\mathrm{ran}}(H; Q) \coloneqq \sup_{\substack{f\in H\\ \|f\|_{H}\le 1}}\EE_{\omega}\left| I(f)-Q^{\omega}(f)\right|. \]
In what follows, we evaluate this randomized error specifically for our fractional Sobolev space $H = H_{\alpha,\bsgamma,d}^{\mathrm{sob}}$ equipped with the norm $\|\cdot\|_{H} = \|\cdot\|_{\alpha,\bsgamma,d}^{\mathrm{sob}}$.

In this context, building upon a classical work by Bakhvalov \cite{B61}, \cite{KKNU19} proved the non-constructive existence of randomized rank-1 lattice rules that achieve a randomized error of order $\mathcal{O}(n^{-\alpha-1/2+\varepsilon})$ for any $\varepsilon>0$ in the weighted Korobov spaces. Fully constructive randomized algorithms have recently been developed; see, e.g., \cite{DGS22,G26,GK26,KNW23}. In what follows, we focus on two different constructive approaches for the half-period cosine space and transfer their results to our novel fractional Sobolev spaces: the randomized fast CBC algorithm from \cite{DGS22} and the robust median-of-means lattice rules from \cite{GK26}.

The randomized CBC algorithm, proposed in \cite{DGS22} and described in Algorithm~\ref{alg:randomized_cbc_lattice}, chooses the number of points $n$ uniformly at random from the set of prime numbers
\[ \PP_{n_{\max}}\coloneqq \left\{ \lceil n_{\max}/2\rceil  <n\le n_{\max}\mathrel{\Big|} \text{$n$ is prime}\right\}, \]
for a prescribed integer $n_{\max} \ge 2$. Conditioned on the chosen $n$, the components of the generating vector $\bsz$ are selected iteratively in a randomized CBC fashion with a threshold parameter $\tau \in (0,1)$.

\begin{algorithm}[H]
\caption{Randomized CBC construction for randomized rank-1 lattice rules}
\label{alg:randomized_cbc_lattice}
\begin{algorithmic}[1]
\REQUIRE A positive integer $n_{\max}$, dimension $d$, smoothness parameter $1/2<\alpha <2$, rescaled weights $\bsGamma_{\alpha}\coloneqq \underline{\kappa}_{\alpha}^{-1/2}\bsgamma$, and a threshold parameter $\tau\in (0,1)$.
\STATE Randomly pick a prime $n \in \PP_{n_{\max}}$ with uniform distribution.
\STATE Set $z_1 = 1$.
\FOR{$s = 2$ \TO $d$}
    \STATE Find a subset $Z_s \subset \{1, \ldots, n-1\}$ of size $|Z_s|=\lceil \tau (n-1)\rceil$ such that
    \[
        E_{n,s}(z^{(1)}_s \mid (z_1, \ldots, z_{s-1})) \le E_{n,s}(z^{(2)}_s \mid (z_1, \ldots, z_{s-1}))
    \]
    for all $z^{(1)}_s \in Z_s$ and all $z^{(2)}_s \in \{1, \ldots, n-1\} \setminus Z_s$.
    \STATE Randomly pick $z_s \in Z_s$ with uniform distribution.
\ENDFOR
\RETURN $\bsz = (z_1, \ldots, z_d)$ and $n$.
\end{algorithmic}
\end{algorithm}

In line~4 of Algorithm~\ref{alg:randomized_cbc_lattice}, we need to arrange the integers $1,\ldots,n-1$ such that the corresponding value $E_{n,s}$ is listed in ascending order. This arrangement is not necessarily unique when some of the integers yield the same value of $E_{n,s}$. However, we can always make the ordering unique by arranging such integers themselves in ascending order.

Now, our randomized integration rule is given by
\[
    Q_{n_{\max},\tau}(f) \coloneqq I(f;P^{\varphi}_{n,\bsz}),
\]
with $n$ and $\bsz$ being randomly drawn according to Algorithm~\ref{alg:randomized_cbc_lattice}. Due to the norm equivalence, the error guarantee proven in \cite[Theorem~4.3]{DGS22} for the half-period cosine space directly yields the randomized error bound in our novel fractional Sobolev space.

\begin{theorem}\label{thm:integration_randomized_cbc}
    Let $1/2 < \alpha < 2$, $\tau \in (0, 1)$ and $\bsgamma=(\gamma_1,\ldots,\gamma_d)$ with $0<\gamma_j\le \underline{\kappa}_{\alpha}^{1/2}$. Assume 
    \begin{align}\label{eq:minimal_size_n}
        n_{\max}\ge 2\inf_{1/2\le \lambda<\alpha}\frac{V_{d}(\alpha,\lambda,\bsgamma)-1}{1-\tau},
    \end{align}
    where $V_{d}(\alpha,\lambda,\bsgamma)$ is defined as in Theorem~\ref{thm:integration_deterministic}. Let $Q_{n_{\max},\tau}$ be the randomized tent-transformed lattice rule with $n$ and $\bsz$ constructed by Algorithm~\ref{alg:randomized_cbc_lattice}. Then, the randomized error in the fractional Sobolev space $H_{\alpha,\bsgamma,d}^{\mathrm{sob}}$ satisfies
    \[
        e^{\mathrm{ran}}(H_{\alpha,\bsgamma,d}^{\mathrm{sob}}; Q_{n_{\max},\tau}) \le C_{\lambda,\delta,\tau}\frac{\left(V_{d}(\alpha,\lambda,\bsgamma)-1\right)^{\lambda-\delta}}{n_{\max}^{\lambda+1/2-\delta}},
    \]
    for any $1/2< \lambda<\alpha$ and $0<\delta<\min(\lambda-1/2,1)$ with a constant $C_{\lambda,\delta,\tau}>0$. 
\end{theorem}

Regarding the dependence of the error bound on the dimension $d$, the same observation as in Remark~\ref{rem:tractability_deterministic} applies to the current randomized setting. To be more precise, under the summability condition $\sum_{j=1}^{\infty} \gamma_j^{1/\alpha} < \infty$, the error bound can be further bounded independently of $d$, and the randomized tent-transformed lattice rule $Q_{n_{\max},\tau}$ achieves strong polynomial tractability with an optimal randomized convergence rate close to $\Ocal(n_{\max}^{-\alpha - 1/2})$.

Although the randomized CBC algorithm described above achieves the near-optimal randomized convergence rate, it requires the smoothness parameter $\alpha$ and the weights $\bsgamma$ as inputs. To construct a more robust algorithm without needing to specify these parameters beforehand, following \cite{GK26}, we introduce the median-of-means estimator in Algorithm~\ref{alg:median_of_means}. In what follows, let $h \colon \NN \to [1,\infty)$ be an arbitrary function satisfying $h(n) \to \infty$ as $n \to \infty$. In practice, one considers a function that grows extremely slowly, such as $h(n) = \max(1, \log n)$ or $h(n) = \max(1, \log\log n)$; see \cite[Remark~11]{GK26}.

\begin{algorithm}[H]
\caption{Median-of-means tent-transformed rank-1 lattice rules}
\label{alg:median_of_means}
\begin{algorithmic}[1]
\REQUIRE A positive integer $n_{\max}$, dimension $d$, a growth function $h \colon \NN \to [1,\infty)$, and an integrand $f \colon [0,1]^d \to \RR$.
\STATE Set $R = 2\lceil h(n_{\max})\log_2 n_{\max}\rceil + 1$.
\FOR{$r = 1$ \TO $R$}
    \STATE Independently pick a prime $n_r \in \PP_{n_{\max}}$ uniformly at random and a generating vector $\bsz_r \in \{1,\ldots,n_r-1\}^d$ uniformly at random.
    \STATE Compute $Q_{n_r,\bsz_r}(f) \coloneqq I(f;P^{\varphi}_{n_r,\bsz_r})$.
\ENDFOR
\RETURN $Q_{n_{\max},h,\mathrm{med}}(f) \coloneqq \operatorname{median}\left( Q_{n_1,\bsz_1}(f), \ldots, Q_{n_R,\bsz_R}(f) \right)$.
\end{algorithmic}
\end{algorithm}

Notably, Algorithm~\ref{alg:median_of_means} does not require any information on $\alpha$ or $\bsgamma$. Nevertheless, similarly to Theorem~\ref{thm:integration_randomized_cbc}, the following randomized error bound holds, translated from \cite[Theorem~10]{GK26}:
\begin{theorem}\label{thm:integration_randomized_median}
    Let $1/2 < \alpha < 2$ and $\bsgamma=(\gamma_1,\ldots,\gamma_d)$ with $0<\gamma_j\le \underline{\kappa}_{\alpha}^{1/2}$. Assume 
    \eqref{eq:minimal_size_n} with $\tau=31/32$. For a growth function $h \colon \NN \to [1,\infty)$, let $Q_{n_{\max},h,\mathrm{med}}$ be the median-of-means tent-transformed lattice rule given by Algorithm~\ref{alg:median_of_means}. Then, for any $1/2< \lambda<\alpha$, any $0<\delta<\min(\lambda-1/2,1)$, and all $n_{\max}$ with $h(n_{\max})\ge \lambda+1/2$, the randomized error in the fractional Sobolev space $H_{\alpha,\bsgamma,d}^{\mathrm{sob}}$ satisfies
    \[
        e^{\mathrm{ran}}(H_{\alpha,\bsgamma,d}^{\mathrm{sob}}; Q_{n_{\max},h,\mathrm{med}}) \le C_{\lambda,\delta}\frac{\left(V_{d}(\alpha,\lambda,\bsgamma)-1\right)^{\lambda}}{n_{\max}^{\lambda+1/2-\delta}},
    \]
    where $V_{d}(\alpha,\lambda,\bsgamma)$ is defined as in Theorem~\ref{thm:integration_deterministic}, with a constant $C_{\lambda,\delta}>0$. 
\end{theorem}

\begin{remark}\label{rem:gamma_bounded}
    In Theorems~\ref{thm:integration_randomized_cbc} and \ref{thm:integration_randomized_median}, we assume $0 < \gamma_j \le \underline{\kappa}_{\alpha}^{1/2}$ for all $j=1, \ldots, d$, which ensures $r_{\alpha,\underline{\kappa}_{\alpha}^{-1/2}\bsgamma}(\bsk) \ge 1$ for any $\bsk \in \NN_0^d$ and thereby simplifies the technical analysis. However, this result can be readily extended to general weight parameters without the restriction $\gamma_j \le \underline{\kappa}_{\alpha}^{1/2}$ by appropriately adjusting the constant factor.
\end{remark}

%%%%%%%%%%%%%%%%%%%%%%%%%%%%%%%%%%%%%%%%%%%%%%%%%%
%%%%%%%%%%%%%%%%%%%%%%%%%%%%%%%%%%%%%%%%%%%%%%%%%%
\section{Function approximation}\label{sec:approximation}

This section is devoted to multivariate function approximation over $[0,1]^d$. The key strategy employed here is ubiquitous in the literature (see, e.g., \cite{CG26,CGK26,K18,K19,KV19,KSW06,KWW09,PGK25,PKG25}): given an orthonormal system and the corresponding infinite series expansion of a target function $f$, we truncate the series into a finite sum and subsequently estimate each coefficient in the sum using a quadrature rule. 

In what follows, in addition to the half-period cosine basis, we also utilize the standard complex Fourier basis $\{e^{2\pi \mi \bsk \cdot \bsx} \mid \bsk\in \ZZ^d\}$, where the $\bsk$-th Fourier coefficient of an integrable function $g$ on $[0,1]^d$ is denoted by 
\[ \widehat{g}(\bsk) \coloneqq \int_{[0,1]^d} g(\bsx) e^{-2\pi \mi \bsk \cdot \bsx} \rd \bsx.\]
The tent transformation provides the following important connection between the half-period cosine coefficients and the Fourier coefficients.

\begin{lemma}\label{lem:cosine_to_Fourier}
    Let $f \in L_2([0,1]^d)$ have an absolutely convergent half-period cosine series. For any $\bsk \in \NN_0^d$ and $\bssigma \in \{\pm 1\}^d$, we have
    \[ 
        \widehat{f\circ \varphi}(\bssigma (\bsk)) = \frac{\widetilde{f}(\bsk)}{2^{|\bsk|_0/2}}, 
    \]
    where $\bssigma(\bsk) \coloneqq (\sigma_1 k_1, \ldots, \sigma_d k_d)$.
\end{lemma}
\begin{proof}
    In this proof, let us write $u = \supp(\bsk)$. The $\bssigma(\bsk)$-th Fourier coefficient of $f \circ \varphi$ is evaluated as
    \begin{align*}
        \widehat{f \circ \varphi}(\bssigma(\bsk)) & = \int_{[0,1]^d} (f \circ \varphi)(\bsx) e^{-2\pi \mi \bssigma(\bsk) \cdot \bsx} \rd \bsx \\
        & = \int_{[0,1]^d} (f \circ \varphi)(\bsx) \prod_{j\in u} \left( \cos(2\pi k_j x_j) - \mi \sigma_j \sin(2\pi k_j x_j) \right) \rd \bsx.
    \end{align*}
    By the definition of the tent transformation $\varphi$, the function $(f \circ \varphi)$ is even with respect to $x_j = 1/2$ for all $j=1,\ldots,d$. Since $\sin(2\pi k_j x_j)$ with $k_j>0$ is always an odd function with respect to $x_j = 1/2$, all terms involving the sine function vanish upon integration. Thus, the above integral simplifies to
    \[ 
        \widehat{f \circ \varphi}(\bssigma(\bsk)) = \int_{[0,1]^d} (f \circ \varphi)(\bsx) \prod_{j\in u} \cos(2\pi k_j x_j) \rd \bsx = \frac{1}{2^{|u|/2}} \int_{[0,1]^d} (f \circ \varphi)(\bsx) \phi_{2\bsk}(\bsx) \rd \bsx.
    \]
    Since $\varphi$ is a measure-preserving map on $[0,1]^d$ and $\phi_{2\bsk}(\bsx) = (\phi_{\bsk} \circ \varphi)(\bsx)$ holds componentwise, we obtain
    \[
        \widehat{f \circ \varphi}(\bssigma(\bsk)) = \frac{1}{2^{|u|/2}} \int_{[0,1]^d} f(\bsx) \phi_{\bsk}(\bsx) \rd \bsx = \frac{1}{2^{|u|/2}} \widetilde{f}(\bsk).
    \]
    This completes the proof.
\end{proof}

In the present context, due to the norm equivalence between our weighted unanchored fractional Sobolev space $H_{\alpha,\bsgamma,d}^{\mathrm{sob}}$ and the weighted half-period cosine space $H_{\alpha,\bsgamma,d}^{\mathrm{cos}}$, the half-period cosine system $\{\phi_{\bsk}\mid \bsk \in \NN_0^d\}$ is a natural and optimal choice for the underlying orthonormal basis of $L_2([0,1]^d)$. Moreover, by Remark~\ref{rem:absolute_convergence}, every $f\in H_{\alpha,\bsgamma,d}^{\sob}$ has an absolutely and uniformly convergent half-period cosine expansion. Together with Lemma~\ref{lem:cosine_to_Fourier}, this also implies the absolute convergence of the Fourier series of $f\circ\varphi$. For a target function $f \in H_{\alpha,\bsgamma,d}^{\mathrm{sob}}$ and a finite index set $\Acal_d \subset \NN_0^d$, Lemma~\ref{lem:cosine_to_Fourier} leads to
\begin{align}\label{eq:series_trunction}
    f(\bsx) = \sum_{\bsk \in \NN_0^d} \widetilde{f}(\bsk) \phi_{\bsk}(\bsx) \approx \sum_{\bsk \in \Acal_d} \widetilde{f}(\bsk) \phi_{\bsk}(\bsx) = \sum_{\bsk \in \Acal_d} \left( \frac{2^{|\bsk|_0/2}}{2^d}\sum_{\bssigma\in \{\pm 1\}^d}\widehat{f\circ \varphi}(\bssigma (\bsk))\right)\phi_{\bsk}(\bsx).
\end{align} 
For a subset $u\subseteq \{1,\ldots,d\}$, let us write
\[ 
    \Acal_{d,u} = \left\{ \bsk_u \in \NN^{|u|} \mid (\bsk_u,\bszero)\in \Acal_d\right\}, 
\]
which provides the following partition of the index set:
\[ 
    \Acal_d = \bigcup_{u\subseteq \{1,\ldots,d\}} \left\{ (\bsk_u,\bszero) \mid \bsk_u \in \Acal_{d,u} \right\}. 
\]
Moreover, we introduce the extended and symmetrized index sets of $\Acal_{d,u}$ and $\Acal_d$, defined by
\[ 
    \Acal_{d,u}^{\sym} = \left\{ \bssigma_u(\bsk_u) \in (\ZZ\setminus \{0\})^{|u|} \mid \bssigma_u\in \{\pm 1\}^{|u|}, \bsk_u\in \Acal_{d,u}\right\} 
\]
and 
\[ 
    \Acal_d^{\sym} = \bigcup_{u\subseteq \{1,\ldots,d\}} \left\{ (\bsk_u,\bszero) \mid \bsk_u \in \Acal_{d,u}^{\sym} \right\}, 
\]
respectively. Then, the right-most side of \eqref{eq:series_trunction} equals
\[ 
    \sum_{u\subseteq \{1,\ldots,d\}}\frac{1}{2^{|u|/2}}\sum_{\bsk_u\in \Acal_{d,u}^{\sym}}\widehat{f\circ \varphi}(\bsk_u,\bszero)\phi_{|\bsk_u,\bszero|}(\bsx) = \sum_{\bsk\in \Acal_{d}^{\sym}}\frac{\widehat{f\circ \varphi}(\bsk)}{2^{|\bsk|_0/2}}\phi_{|\bsk|}(\bsx),
\]
where, for any vector $\bsk \in \ZZ^d$, we denote the componentwise absolute value by $|\bsk| \coloneqq (|k_1|, \ldots, |k_d|) \in \NN_0^d$. Thus, the problem reduces to how to estimate the Fourier coefficients $\widehat{f\circ \varphi}(\bsk)$ for all $\bsk\in \Acal_{d}^{\sym}$.

%%%%%%%%%%%%%%%%%%%%%%%%%%%%%%%%%%%%%%%%%%%%%%%%%%
\subsection{Multiple lattice-based algorithm}

%%%%%%%%%%%%%%%%%%%%%%%%%%%%%%%%%%%%%%%%%%%%%%%%%%
\subsubsection{Overview}
As explained in Section~\ref{sec:intro}, using a single rank-1 lattice point set $P_{n,\bsz}$ as fixed sampling nodes is often insufficient to achieve the optimal rate of convergence in the weighted Korobov space. The same issue applies to the weighted half-period cosine space (and thus, the weighted fractional Sobolev space as well) when approximating functions by a single tent-transformed rank-1 lattice point set \cite{CKNS16}. 

The key obstacle is the so-called \emph{aliasing phenomenon}. If there exist two distinct frequencies $\bsk_1,\bsk_2\in \Acal_d^{\sym}$ such that $\bsk_1-\bsk_2\in P_{n,\bsz}^{\perp}$ (recall the definition of the dual lattice given in Section~\ref{subsec:lattice}), a single rank-1 lattice rule cannot distinguish between these frequencies. In our context, when we approximate $\widehat{f\circ \varphi}(\bsk)$ by a rank-1 lattice rule as:
\[ 
    \widehat{(f\circ \varphi)}_{n,\bsz}(\bsk)\coloneqq I((f\circ \varphi)\cdot e^{-2\pi \mi \bsk\cdot (\cdot)}; P_{n,\bsz}) = \frac{1}{n}\sum_{i=0}^{n-1} f\circ \varphi(\bsy_i)e^{-2\pi \mi \bsk\cdot \bsy_i},
\]
it holds that
\begin{align}\label{eq:aliasing}
    \widehat{(f\circ \varphi)}_{n,\bsz}(\bsk_1) = \sum_{\bsl\in P_{n,\bsz}^{\perp}}\widehat{f\circ \varphi}(\bsk_1+\bsl) = \sum_{\bsl\in P_{n,\bsz}^{\perp}}\widehat{f\circ \varphi}(\bsk_2+\bsl) = \widehat{(f\circ \varphi)}_{n,\bsz}(\bsk_2),
\end{align}
where the second equality follows from the group structure of the dual lattice. The existence of a pair of such aliased frequencies within the index set $\Acal_{2}^{\sym}=\{-\lfloor\sqrt{n}\rfloor,\ldots,-1,0,1,\ldots,\lfloor\sqrt{n}\rfloor\}^2$ was proven by the pigeonhole principle in a classical work of Smolyak \cite{S60}. This aliasing phenomenon yields the non-optimality of the convergence rate for any single rank-1 lattice-based approximation \cite{BKUV17}.

To circumvent this fundamental limitation and successfully disentangle the aliased Fourier coefficients, here we employ an algorithm based on multiple rank-1 lattice rules \cite{K18,K19,KV19}; see also the monograph \cite[Chapter~15]{DKP22} and the recent work by the authors \cite{CG26}. We note that a different approach based on a lattice with multiple shifts and a least-squares algorithm was recently studied in \cite{DD26} for the half-period cosine spaces. While their result also applies to our fractional Sobolev spaces, our aim here is to provide an $L_{\infty}$ approximation error bound with better dimension-dependence of the implied constant factor.

In the multiple lattice-based algorithm, the Fourier coefficients for all $\bsk\in \Acal_d^{\sym}$ are approximated using a collection of $L$ distinct rank-1 lattice point sets $P_{n_1,\bsz_1},\ldots,P_{n_L,\bsz_L}$. The number $L$ and the specific lattice point sets are chosen appropriately depending on the index set $\Acal_d^{\sym}$. At least one suitable lattice rule among these $L$ candidates is assigned to each frequency $\bsk\in \Acal_d^{\sym}$ such that aliasing with any other frequency $\bsl\in \Acal_d^{\sym}\setminus \{\bsk\}$ does not occur, meaning that $\bsk-\bsl \notin P_{n_t,\bsz_t}^{\perp}$ for at least one $t\in \{1,\ldots,L\}$.

For $t=1,\ldots,L$, let $\Acal_{d,t}^{\sym}\subseteq \Acal_{d}^{\sym}$ be the subset of frequencies whose Fourier coefficients are estimated using the $t$-th lattice point set $P_{n_t,\bsz_t}$. According to the above explanation, we define this subset $\Acal_{d,t}^{\sym}$ as
\[ 
    \Acal_{d,t}^{\sym} = \left\{ \bsk\in \Acal_d^{\sym} \mid \bsk-\bsl\notin P_{n_t,\bsz_t}^{\perp}\; \text{for all $\bsl\in \Acal_d^{\sym}\setminus \{\bsk\}$}\right\}. 
\]
Here, we require the covering condition
\[ 
    \bigcup_{t=1}^{L}\Acal_{d,t}^{\sym} = \Acal_{d}^{\sym},
\]
although the subsets $\Acal_{d,t}^{\sym}$ are not necessarily mutually disjoint.
Using an indicator function $\bsone_{\bullet}$, we write
\[ 
    U(\bsk) = \sum_{t=1}^{L}\bsone_{\bsk\in \Acal_{d,t}^{\sym}}
\]
for each $\bsk\in \Acal_d^{\sym}$, which counts multiplicity, i.e., how many lattices are assigned to the frequency $\bsk$. Then, the multiple tent-transformed lattice-based algorithm is given by
\begin{align}\label{eq:multiple_lattice_approx}
    A^{\mathrm{mult}}(f)(\bsx)=\sum_{t=1}^{L}\sum_{\bsk\in \Acal_{d,t}^{\sym}}\frac{\widehat{(f\circ \varphi)}_{n_t,\bsz_t}(\bsk)}{U(\bsk) 2^{|\bsk|_0/2}}\phi_{|\bsk|}(\bsx).
\end{align}

%%%%%%%%%%%%%%%%%%%%%%%%%%%%%%%%%%%%%%%%%%%%%%%%%%
\subsubsection{Construction}

We now discuss how to construct the collection of $L$ rank-1 lattice point sets for a general finite index set $\Acal_{d}^{\sym}$. Although we assume that $\Acal_{d}^{\sym}$ is a component-wise symmetric set, the result shown here can be extended to a more general index set. The case $|\Acal_d^{\sym}|=1$ is trivial. Hence, throughout this subsection, we assume that $|\Acal_d^{\sym}|\ge 2$. To maintain a reasonable ratio of the number of sampling points to the size of the index set, we adopt a probabilistic construction strategy proposed by K\"ammerer \cite[Algorithm~1]{K19}. In this approach, we use a common prime number of points $n$ for all $L$ lattices, i.e., $n_1 = \cdots = n_L = n$.

For a given frequency set $\Acal_d^{\sym}$, we first define the maximum componentwise span as
\begin{align}\label{eq:extension_length}
    N_{\Acal_d^{\sym}} \coloneqq \max_{j=1,\ldots,d}\left(\max_{\bsk\in \Acal_d^{\sym}} k_j - \min_{\bsh\in \Acal_d^{\sym}} h_j \right).
\end{align}
To avoid trivial aliasing in individual coordinates, the prime $n$ must be chosen strictly larger than $N_{\Acal_d^{\sym}}$. Note that for a single prime $n$, the condition $\bsk-\bsl \in P_{n,\bsz_t}^{\perp}$ is equivalent to $\bsk\cdot\bsz_t \equiv \bsl\cdot\bsz_t \pmod n$. The following theorem, proven in \cite[Theorem~3.2, Corollary~3.3 and Theorem~3.4]{K19}, guarantees that a relatively small number of randomly generated lattices is sufficient to cover the entire index set $\Acal_d^{\sym}$ with high probability.

\begin{theorem}[\cite{K19}]\label{thm:probabilistic_construction}
    Let $\Acal_d^{\sym}$ be a finite frequency set with maximum span $N_{\Acal_d^{\sym}}$ defined in \eqref{eq:extension_length}. For given $\delta\in (0,1)$ and $c>1$, let 
    \[ c'=\max\left\{ c,\frac{N_{\Acal_d^{\sym}}}{|\Acal_d^{\sym}|-1}\right\}, \]
    and determine two numbers
    \begin{align*}
        \tau \ge c'(|\Acal_d^{\sym}|-1), \qquad
        L = \left\lceil \left(\frac{c'}{c'-1}\right)^2 \frac{ \log |\Acal_d^{\sym}| - \log \delta}{2}\right\rceil.
    \end{align*}
    Let $n$ be the smallest prime strictly greater than $\tau$. If we choose the generating vectors $\bsz_1, \ldots, \bsz_L \in \{0,\ldots,n-1\}^d$ independently and uniformly at random, then with probability at least $1-\delta$, the covering condition $\bigcup_{t=1}^L \Acal_{d,t}^{\sym} = \Acal_d^{\sym}$ holds.
\end{theorem}

Based on this theorem, the corresponding probabilistic algorithm for constructing the multiple rank-1 lattices is summarized in Algorithm~\ref{alg:multiple_lattice}. It should be noted that this algorithm satisfies the covering condition with a probability of at least $1-\delta$. Since this condition can be explicitly verified in each run, one can simply repeat the procedure until the condition is met. Once a valid set of $L$ rank-1 lattices is constructed and fixed, the subsequent function approximation scheme becomes entirely deterministic.

\begin{algorithm}[H]
\caption{Probabilistic construction of multiple rank-1 lattices.}
\label{alg:multiple_lattice}
\begin{algorithmic}[1]
\REQUIRE Index set $\Acal_d^{\sym}$, maximum span $N_{\Acal_d^{\sym}}$, failure probability $\delta \in (0, 1)$, and constant $c > 1$.
\STATE Compute $c' = \max\left\{ c, \frac{N_{\Acal_d^{\sym}}}{|\Acal_d^{\sym}|-1} \right\}$ and choose $\tau = c'(|\Acal_d^{\sym}|-1)$.
\STATE Calculate the required number of lattices:
\begin{align*}
    L \coloneqq \left\lceil \left( \frac{c'}{c'-1} \right)^2 \frac{\log | \Acal_d^{\sym}| - \log \delta}{2} \right\rceil.
\end{align*}
\STATE Find the smallest prime $n > \tau$.
\FOR{$t = 1$ \TO $L$}
    \STATE Choose a generating vector $\bsz_t \in \{0, \ldots, n-1\}^d$ uniformly at random.
    \STATE Identify the alias-free frequencies $\Acal_{d,t}^{\sym} \subseteq \Acal_d^{\sym}$ for the lattice $P_{n,\bsz_t}$:
    \begin{align*}
        \Acal_{d,t}^{\sym} \coloneqq \Big\{\bsk \in \Acal_d^{\sym} \ \Big|\ \bsk\cdot\bsz_t\not\equiv \bsl\cdot\bsz_t \pmod n,\ \forall\bsl\in \Acal_d^{\sym}\setminus\{\bsk \}\Big\}.
    \end{align*}
\ENDFOR
\ENSURE The prime $n$, the set of generating vectors $\{\bsz_1,\ldots,\bsz_L\}$, and the subsets $\{\Acal_{d,1}^{\sym}, \ldots, \Acal_{d,L}^{\sym}\}$.
\end{algorithmic}
\end{algorithm}

We briefly analyze the computational complexity of Algorithm~\ref{alg:multiple_lattice}. We note that, by Bertrand's postulate, the required prime $n$ is guaranteed to exist in the interval $(\tau, 2\tau]$; finding the prime $n$ in line~3 takes $\mathcal{O}(\tau \log \log \tau)$ operations using the sieve of Eratosthenes.  In each of the $L$ iterations of the for-loop, computing the dot products $\bsk \cdot \bsz_t \pmod n$ for all $\bsk \in \Acal_d^{\sym}$ requires $\mathcal{O}(d |\Acal_d^{\sym}|)$ operations. Subsequently, identifying the unique, alias-free frequencies to construct $\Acal_{d,t}^{\sym}$ can be done by sorting the computed dot products, which takes $\mathcal{O}(|\Acal_d^{\sym}| \log |\Acal_d^{\sym}|)$ operations. Therefore, the overall computational complexity of Algorithm~\ref{alg:multiple_lattice} is bounded by
\begin{align*}
    \mathcal{O}\Big( \tau \log \log \tau + L |\Acal_d^{\sym}| (d + \log |\Acal_d^{\sym}|) \Big).
\end{align*}
Since $L$ scales logarithmically with $|\Acal_d^{\sym}|$, the cost associated with the random generation and alias-checking is well controlled and typically dominated by the term scaling with $|\Acal_d^{\sym}| \log^2 |\Acal_d^{\sym}|$.

%%%%%%%%%%%%%%%%%%%%%%%%%%%%%%%%%%%%%%%%%%%%%%%%%%
\subsubsection{Weighted hyperbolic cross}

Having established the multiple lattice algorithm for a general index set, we now specify the frequency set $\Acal_d^{\sym}$ as the weighted hyperbolic cross. In our context, we first introduce the positive quadrant of the weighted hyperbolic cross with radius $M \ge 1$, defined as
\begin{align}\label{eq:weight_hyper_positive}
    \Acal_{d} \coloneqq \Acal_{d,\alpha,\bsgamma}(M)= \left\{\bsk \in \mathbb{N}^d_0 \mid r_{\alpha,\bsGamma_{\alpha}} (\bsk)\leq M  \right\},
\end{align}
where $1/2<\alpha<2$ is the smoothness parameter, and $\bsGamma_{\alpha}= \underline{\kappa}_{\alpha}^{-1/2}\bsgamma$ denotes the set of rescaled weights. Its symmetrization, which we use as the target frequency set $\Acal_d^{\sym}$ in the multiple lattice algorithm, is naturally given by
\begin{align}\label{eq:weight_hyper_sym}
    \Acal_d^{\sym} \coloneqq \Acal_{d,\alpha,\bsgamma}^{\sym}(M) = \left\{\bsk \in \mathbb{Z}^d \mid r_{\alpha,\bsGamma_{\alpha}} (|\bsk|)\leq M  \right\}.
\end{align}

To apply Algorithm~\ref{alg:multiple_lattice}, it is essential to evaluate the cardinality $|\Acal_d^{\sym}|$ and the maximum span $N_{\Acal_d^{\sym}}$. The size of the weighted hyperbolic cross can be bounded as follows. We refer to \cite[Lemma~13.1]{DKP22} or \cite[Lemma~2.4]{CG26} for the proof.

\begin{lemma}\label{lem:size_hyperbolic_cross}
    For $1/2<\alpha<2$, $d \in \NN$, a collection of positive weights $\bsgamma=(\gamma_1,\ldots,\gamma_d)$, and $M \ge 1$, let $\Acal_d^{\sym}$ be defined as in \eqref{eq:weight_hyper_sym}. Then, for any $0<\lambda <\alpha$, we have
    \begin{align*}
        |\Acal_d^{\sym}| \leq M^{1/\lambda} V_{d}(\alpha,\lambda,\bsgamma),
    \end{align*}
    where $V_{d}(\alpha,\lambda,\bsgamma)$ is defined as in Theorem~\ref{thm:integration_deterministic}.
\end{lemma}
The cardinality of the positive quadrant $|\Acal_{d,\alpha,\bsgamma}(M)|$ is similarly bounded by
    \begin{align*}
        |\Acal_{d,\alpha,\bsgamma}(M)| \leq M^{1/\lambda} \prod_{j=1}^{d}\left(1+\Gamma_{j,\alpha}^{1/\lambda}\zeta(\alpha/\lambda)\right),
    \end{align*}
for any $0<\lambda<\alpha$.

Following the analysis given in \cite[Section~4.3]{CG26}, it holds that $N_{\Acal_d^{\sym}} \le |\Acal_d^{\sym}| - 1$ for any radius $M\ge 1$. 
This implies that the parameter $c'$ in Algorithm~\ref{alg:multiple_lattice} simply reduces to $c$. Thus, as in \cite[Corollary~A.2]{CG26}, the total number of function evaluations $n_{\mathrm{tot}}$ is bounded by
\begin{align}\label{eq:total_n_approximation}
    n_{\mathrm{tot}} & = nL\le 2\tau \left\lceil \left( \frac{c}{c-1} \right)^2 \frac{\log | \Acal_d^{\sym}| - \log \delta}{2} \right\rceil \notag \\
    & \le 2c(|\Acal_d^{\sym}| - 1)\left( \left( \frac{c}{c-1} \right)^2 \frac{\log | \Acal_d^{\sym}| - \log \delta}{2} +1 \right) \notag \\
    & \le C_{c,\delta} |\Acal_d^{\sym}| \log |\Acal_d^{\sym}| \le C_{c,\delta,\beta}|\Acal_d^{\sym}|^{1+\beta} \le C_{c,\delta,\beta}M^{(1+\beta)/\lambda}(V_{d}(\alpha,\lambda,\bsgamma))^{1+\beta},
\end{align}
for any $0<\lambda<\alpha$ and $0<\beta<1$, where the third inequality follows from the elementary inequality $\log x\le x^\beta/\beta$, and the last inequality stems from Lemma~\ref{lem:size_hyperbolic_cross}.
 
%%%%%%%%%%%%%%%%%%%%%%%%%%%%%%%%%%%%%%%%%%%%%%%%%%
\subsection{Error Analysis}

Finally, we prove an upper bound on the worst-case $L_{\infty}$ error of the multiple tent-transformed lattice algorithm $A^{\mathrm{mult}}$ given in \eqref{eq:multiple_lattice_approx}. Here, the worst-case $L_{\infty}$ error is defined by
\[ e_{\infty}^{\mathrm{wor}}(H_{\alpha,\bsgamma,d}^{\mathrm{sob}}; A^{\mathrm{mult}}) \coloneqq \sup_{\substack{f\in H_{\alpha,\bsgamma,d}^{\sob}\\ \|f\|_{\alpha,\bsgamma,d}^{\sob}\le 1}} \|f-A^{\mathrm{mult}}(f)\|_{L_{\infty}}, \]
where the $L_{\infty}$ error for an individual function is simply given by
\[ \|f-A^{\mathrm{mult}}(f)\|_{L_{\infty}} =\sup_{\bsx\in [0,1]^d}\left|f(\bsx)-A^{\mathrm{mult}}(f)(\bsx)\right|.\]

\begin{theorem}\label{thm:approximation}
    For $1/2<\alpha<2$, $d \in \NN$, a collection of positive weights $\bsgamma=(\gamma_1,\ldots,\gamma_d)$, and $M \ge 1$, let $\Acal_d^{\sym}$ be defined as in \eqref{eq:weight_hyper_sym}. Consider the multiple tent-transformed lattice algorithm   \eqref{eq:multiple_lattice_approx}, where the number of lattices $L$ and the lattice parameters $(n_t,\bsz_t)$ for $t=1,\ldots,L$ are chosen by Algorithm~\ref{alg:multiple_lattice} with inputs $0<\delta<1$ and $c>1$, so that the covering condition $\bigcup_{t=1}^L \Acal_{d,t}^{\sym} = \Acal_d^{\sym}$ is met.
    Then, we have
    \[ e_{\infty}^{\mathrm{wor}}(H_{\alpha,\bsgamma,d}^{\mathrm{sob}}; A^{\mathrm{mult}}) \le C_{c,\delta,\beta,\lambda}\frac{\left(V_{d}(\alpha,\lambda,\bsgamma)\right)^{\lambda}}{n_{\mathrm{tot}}^{\lambda/(1+\beta)-1/2}}, \]
    for any $1/2<\lambda<\alpha$ and $0<\beta<\min(1,2\lambda-1)$. Here, $n_{\mathrm{tot}}=n_1+\cdots+n_L=nL$ denotes the total number of function evaluations, $C_{c,\delta,\beta,\lambda}>0$ is a constant depending only on $c,\delta,\beta,\lambda$, and $V_{d}(\alpha,\lambda,\bsgamma)$ is defined as in Theorem~\ref{thm:integration_deterministic}.
\end{theorem}

In the proof of Theorem~\ref{thm:approximation}, we need the following two auxiliary results. 

\begin{lemma}\label{lem:multiple_rank-1_property}
Let $\Acal_d^{\sym}$ be the weighted hyperbolic cross defined in \eqref{eq:weight_hyper_sym} with radius $M\ge 1$, and let $\{P_{n_t, \bsz_t}\}_{t=1}^L$ be the multiple rank-1 lattice point sets obtained by Algorithm~\ref{alg:multiple_lattice} with $n_1=\cdots=n_L$, so that the covering condition $\bigcup_{t=1}^L \Acal_{d,t}^{\sym} = \Acal_d^{\sym}$ is met. Then, for any $t\in \{1,\ldots,L\}$ and any non-zero vector $\bsl\in P_{n_t,\bsz_t}^{\perp}\setminus\{ \bszero \}$, the following properties hold:
\begin{enumerate}
    \item For any $\bsk\in \Acal^{\sym}_{d,t}$, we have
        \begin{align*}
            \bsk+\bsl\notin \Acal_d^{\sym}.
        \end{align*}
    \item The sets
        \begin{align*}
            \{\bsk+\bsl \mid \bsl\in   P_{n_t,\bsz_t}^{\perp}\setminus\{ \bszero \} \}
        \end{align*}
        for distinct $\bsk\in \Acal_{d,t}^{\sym}$ are pairwise disjoint.
\end{enumerate}
\end{lemma}

Although our direct reference to this result is \cite[Lemma~4.1]{CG26}, the result itself has been available in the literature (see, for instance, \cite[Proof of Lemma~3.1]{KV19} and \cite[Proof of Lemma~15.6]{DKP22}). The other auxiliary result is shown in \cite[Lemma~2.5]{CG26}:

\begin{lemma}\label{lem:sum_weights_hyperbolic_cross}
    Let $\Acal_d^{\sym}$ be the weighted hyperbolic cross defined in \eqref{eq:weight_hyper_sym} with radius $M\ge 1$. Then, for any $1/2<\lambda<\alpha$, we have
    \[ \sum_{\bsk\in \ZZ^d\setminus \Acal_d^{\sym}}\frac{1}{r_{\alpha,\bsGamma_{\alpha}}^2(|\bsk|)} \le \frac{1}{M^{2-1/\lambda}}\,\frac{8(3-1/\lambda)}{2-1/\lambda}V_{d}(\alpha,\lambda,\bsgamma), \]
    where $V_{d}(\alpha,\lambda,\bsgamma)$ is defined as in Theorem~\ref{thm:integration_deterministic}.
\end{lemma}

Now we are ready to prove Theorem~\ref{thm:approximation}.
\begin{proof}
For any $f\in H_{\alpha,\bsgamma,d}^{\mathrm{sob}}$, consider its half-period cosine expansion. Then it follows from Lemma~\ref{lem:cosine_to_Fourier} that
\begin{align*}
    f(\bsx) = \sum_{\bsk \in \NN_0^d} \widetilde{f}(\bsk) \phi_{\bsk}(\bsx) = \sum_{\bsk \in \ZZ^d} \frac{\widetilde{f}(|\bsk|)}{2^{|\bsk|_0}} \phi_{|\bsk|}(\bsx) = \sum_{\bsk \in \ZZ^d} \frac{\widehat{f\circ \varphi}(\bsk)}{2^{|\bsk|_0/2}} \phi_{|\bsk|}(\bsx).
\end{align*}
Thus, the $L_{\infty}$ error is bounded as
\begin{align*}
\|f-A^{\mathrm{mult}}(f)\|_{L_{\infty}} 
&= \left\|\sum_{\bsk \in \ZZ^d} \frac{\widehat{f\circ \varphi}(\bsk)}{2^{|\bsk|_0/2}} \phi_{|\bsk|}(\bsx) -\sum_{t=1}^{L}\sum_{\bsk\in \Acal_{d,t}^{\sym}}\frac{\widehat{(f\circ \varphi)}_{n_t,\bsz_t}(\bsk)}{U(\bsk) 2^{|\bsk|_0/2}}\phi_{|\bsk|}(\bsx)\right\|_{L_{\infty}} \\
& \le \left\|\sum_{\bsk \in \Acal_d^{\sym}} \frac{\widehat{f\circ \varphi}(\bsk)}{2^{|\bsk|_0/2}} \phi_{|\bsk|}(\bsx) -\sum_{t=1}^{L}\sum_{\bsk\in \Acal_{d,t}^{\sym}}\frac{\widehat{(f\circ \varphi)}_{n_t,\bsz_t}(\bsk)}{U(\bsk) 2^{|\bsk|_0/2}}\phi_{|\bsk|}(\bsx)\right\|_{L_{\infty}} \\
& \quad + \left\|\sum_{\bsk \in \ZZ^d\setminus \Acal_d^{\sym}} \frac{\widehat{f\circ \varphi}(\bsk)}{2^{|\bsk|_0/2}} \phi_{|\bsk|}(\bsx)\right\|_{L_{\infty}}\\
& = \left\|\sum_{t=1}^{L}\sum_{\bsk\in \Acal_{d,t}^{\sym}} \frac{\widehat{(f\circ \varphi)}(\bsk)-\widehat{(f\circ \varphi)}_{n_t,\bsz_t}(\bsk)}{U(\bsk) 2^{|\bsk|_0/2}}\phi_{|\bsk|}(\bsx)\right\|_{L_{\infty}} \\
& \quad + \left\|\sum_{\bsk \in \ZZ^d\setminus \Acal_d^{\sym}} \frac{\widehat{f\circ \varphi}(\bsk)}{2^{|\bsk|_0/2}} \phi_{|\bsk|}(\bsx)\right\|_{L_{\infty}}\\
&\leq \sum_{t=1}^{L}\sum_{\bsk\in \Acal_{d,t}^{\sym}} \left|\widehat{(f\circ \varphi)}(\bsk)-\widehat{(f\circ \varphi)}_{n_t,\bsz_t}(\bsk)  \right|+\sum_{\bsk \in \ZZ^d\setminus \Acal_d^{\sym}} \left|\widehat{(f\circ \varphi)}(\bsk) \right| .
\end{align*}

For every fixed $t\in \{1,\ldots,L\}$, the equality given in \eqref{eq:aliasing} and Lemma~\ref{lem:multiple_rank-1_property} leads to
\begin{align*}
    \sum_{\bsk\in \Acal_{d,t}^{\sym}} \left|\widehat{(f\circ \varphi)}(\bsk)-\widehat{(f\circ \varphi)}_{n_t,\bsz_t}(\bsk)  \right| & = \sum_{\bsk\in \Acal_{d,t}^{\sym}} \left|\sum_{\bsl\in P_{n_t,\bsz_t}^{\perp}\setminus \{\bszero\}}\widehat{f\circ \varphi}(\bsk+\bsl) \right| \\
    & \le \sum_{\bsk\in \Acal_{d,t}^{\sym}} \sum_{\bsl\in P_{n_t,\bsz_t}^{\perp}\setminus \{\bszero\}}\left|\widehat{f\circ \varphi}(\bsk+\bsl) \right| \\
    & \le \sum_{\bsk \in \ZZ^d\setminus \Acal_d^{\sym}} \left|\widehat{(f\circ \varphi)}(\bsk) \right| .
\end{align*}
Therefore, by using Lemma~\ref{lem:cosine_to_Fourier}, the Cauchy-Schwarz inequality, the symmetry of $\Acal_d^{\sym}$ and Lemma~\ref{lem:sum_weights_hyperbolic_cross} in this order, we obtain
\begin{align*}
\|f-A^{\mathrm{mult}}(f)\|_{L_{\infty}} 
& \le (L+1)\sum_{\bsk \in \ZZ^d\setminus \Acal_d^{\sym}} \left|\widehat{(f\circ \varphi)}(\bsk) \right| \\
& = (L+1)\sum_{\bsk \in \ZZ^d\setminus \Acal_d^{\sym}} \frac{\left|\widetilde{f}(|\bsk|) \right|}{2^{|\bsk|_0/2}}\\
& \le (L+1)\left(\sum_{\bsk \in \ZZ^d\setminus \Acal_d^{\sym}} \frac{\left|\widetilde{f}(|\bsk|) \right|^2 r_{\alpha,\bsGamma_{\alpha}}^2(|\bsk|)}{2^{|\bsk|_0}}\right)^{1/2}\left(\sum_{\bsk\in \ZZ^d\setminus \Acal_d^{\sym}}\frac{1}{r_{\alpha,\bsGamma_{\alpha}}^2(|\bsk|)}\right)^{1/2}\\
& = (L+1)\left(\sum_{\bsk \in \NN_0^d\setminus \Acal_d} \left|\widetilde{f}(\bsk) \right|^2 r_{\alpha,\bsGamma_{\alpha}}^2(|\bsk|)\right)^{1/2}\left(\sum_{\bsk\in \ZZ^d\setminus \Acal_d^{\sym}}\frac{1}{r_{\alpha,\bsGamma_{\alpha}}^2(|\bsk|)}\right)^{1/2}\\
& \le (L+1) \|f\|_{\alpha,\bsGamma_{\alpha},d}^{\cos} \left(\sum_{\bsk\in \ZZ^d\setminus \Acal_d^{\sym}}\frac{1}{r_{\alpha,\bsGamma_{\alpha}}^2(|\bsk|)}\right)^{1/2} \\
& \le (L+1) \frac{\|f\|_{\alpha,\bsgamma,d}^{\sob}}{M^{1-1/(2\lambda)}}\,\sqrt{\frac{8(3-1/\lambda)}{2-1/\lambda}}(V_{d}(\alpha,\lambda,\bsgamma))^{1/2}.
\end{align*}

Since the number of lattices $L$ satisfies
\begin{align*}
    L+1 & \le 2L\le C'_{c,\delta}\log |\Acal_d^{\sym}|\le \frac{2C'_{c,\delta}}{\beta}|\Acal_d^{\sym}|^{\beta/2}\\
    & \le \frac{2C'_{c,\delta}}{\beta}M^{\beta/(2\lambda)} (V_{d}(\alpha,\lambda,\bsgamma))^{\beta/2},
\end{align*}
and the total number of function evaluations $n_{\mathrm{tot}}$ is bounded as shown in \eqref{eq:total_n_approximation}, we obtain
\begin{align*}
    e_{\infty}^{\mathrm{wor}}(H_{\alpha,\bsgamma,d}^{\mathrm{sob}}; A^{\mathrm{mult}}) & \le \frac{2C'_{c,\delta}}{\beta}\sqrt{\frac{8(3-1/\lambda)}{2-1/\lambda}}\, \frac{(V_{d}(\alpha,\lambda,\bsgamma))^{(1+\beta)/2}}{M^{1-(1+\beta)/(2\lambda)}} \\
    & \le C_{c,\delta,\beta,\lambda}\frac{\left(V_{d}(\alpha,\lambda,\bsgamma)\right)^{\lambda}}{n_{\mathrm{tot}}^{\lambda/(1+\beta)-1/2}},
\end{align*}
for any $1/2<\lambda<\alpha$ and $0<\beta<\min(1,2\lambda-1)$. This completes the proof.
\end{proof}

\begin{remark}
    By letting $\lambda\to \alpha^{-}$ and $\beta\to 0^{+}$ in Theorem~\ref{thm:approximation}, the corresponding rate of convergence is arbitrarily close to $\Ocal(n_{\mathrm{tot}}^{-\alpha+1/2})$. This rate for $L_{\infty}$ approximation is known to be optimal up to logarithmic factors for the Korobov space \cite{BDSU16}. Although establishing a corresponding lower bound
    is beyond the scope of this paper, it is plausible that the
    lower-bound argument for the Korobov space can be adapted to the
    half-period cosine space. Thus, the optimal rate for the
    half-period cosine space is expected to be $\Ocal(n_{\mathrm{tot}}^{-\alpha+1/2})$ up to logarithmic factors. By the norm equivalence, the multiple tent-transformed lattice-based algorithm achieves the same upper convergence rate in our fractional Sobolev spaces with $1/2<\alpha<2$. This rate is therefore expected to be optimal up to logarithmic factors.

    Moreover, the $L_{\infty}$ error bound is further bounded independently of the dimension $d$ under the summability condition
    \[ \sum_{j=1}^{\infty} \gamma_j^{1/\alpha} < \infty. \]  
    This condition coincides with the sufficient condition obtained for deterministic integration; see Remark~\ref{rem:tractability_deterministic}.
\end{remark}

%%%%%%%%%%%%%%%%%%%%%%%%%%%%%%%%%%%%%%%%%%%%%%%%%%
%%%%%%%%%%%%%%%%%%%%%%%%%%%%%%%%%%%%%%%%%%%%%%%%%%
\section{Numerical experiments}
We conclude this paper by illustrating the theoretical results established in Sections~\ref{sec:integration} and~\ref{sec:approximation} through numerical experiments on multivariate integration and function approximation in our fractional Sobolev spaces.

%%%%%%%%%%%%%%%%%%%%%%%%%%%%%%%%%%%%%%%%%%%%%%%%%%
\subsection{Numerical integration}
We first investigate the empirical performance of the median-of-means tent-transformed rank-1 lattice rule introduced
in Algorithm~\ref{alg:median_of_means}. A more comprehensive comparison, including Algorithms~\ref{alg:cbc_lattice} and~\ref{alg:randomized_cbc_lattice}, as well as other randomized lattice constructions such as those proposed in \cite{G26,KNW23}, is beyond the scope of this paper.

For $\rho\in(0,2)$ and a vector of positive amplitudes
$\bsa=(a_1,\ldots,a_d)$, we consider the product test function
\[ f_{\rho,\bsa}(\bsx)=\prod_{j=1}^{d}\left( 1+a_j g_{\rho}(x_j)\right),\]
where
\[ g_{\rho}(x) = \frac{(\rho+1)(\rho+2)}{2}x^{\rho}\left( 1-\frac{\rho}{\rho+1}x\right)-1. \]
Direct integration gives $I(g_{\rho})=0$ and hence $I(f_{\rho,\bsa})=1$ for every $\rho$ and every choice of $\bsa$. Furthermore, 
\[ g'_{\rho}(x)=\frac{\rho(\rho+1)(\rho+2)}{2}x^{\rho-1}(1-x). \]
Thus, $g'_{\rho}(1)=0$, and, whenever $\rho>1$, we also have $g'_{\rho}(0)=0$, consistently with the boundary constraints arising when $\alpha>3/2$. Since $|\widetilde g_\rho(k)|\asymp k^{-\rho-1}$ as $k\to\infty$, we have, within the range $1/2<\alpha<2$,
\[
    g_\rho\in H_{\alpha,\gamma,1}^{\sob}
    \quad\Longleftrightarrow\quad
    \alpha<\rho+\frac12.
\] Therefore, Theorem~\ref{thm:integration_randomized_median} suggests that, for every $\varepsilon>0$, the expected absolute integration error can decay at a rate arbitrarily close to $\Ocal\left(n_{\max}^{-\rho-1+\varepsilon}\right)$.

In our experiments, we consider $\rho\in \{0.3, 0.8, 1.3\}$. The corresponding limiting decay exponents suggested by the theory are $1.3$, $1.8$, and $2.3$, respectively. For each choice of $\rho$, $d$, $\bsa$, and $n_{\max}\in\{2^m\mid m=5,\ldots,14\}$, we generate $B=50$ independent realizations $Q_{n_{\max},h,\mathrm{med}}^{(1)},\ldots,Q_{n_{\max},h,\mathrm{med}}^{(B)}$ of Algorithm~\ref{alg:median_of_means}. We take $h(n) = \max(1, \log\log n)$. Since the exact integral is known a priori, we estimate the expected absolute error by the sample average
\[ \widehat E_B(n_{\max})\coloneqq \frac{1}{B}\sum_{b=1}^{B}\left| I(f_{\rho,\bsa})-Q_{n_{\max},h,\mathrm{med}}^{(b)}(f_{\rho,\bsa})\right|. \]
For the $b$-th realization, let $n_{\mathrm{tot}}^{(b)}=n^{(b)}_1+\cdots+n^{(b)}_R$ denote the total number of function evaluations. In the figures below, the horizontal axis represents the mean computational cost
\[
    \overline n_{\mathrm{tot}}
    \coloneqq
    \frac{1}{B}
    \sum_{b=1}^{B}
    n_{\mathrm{tot}}^{(b)}.
\]
To compare the observed convergence directly with
Theorem~\ref{thm:integration_randomized_median}, the empirical decay
rates reported below are obtained by a least-squares regression of
$\log\widehat E_B(n_{\max})$ against $\log n_{\max}$ over the five
largest values of $n_{\max}$.

\begin{figure}[t]
    \centering
    \begin{minipage}[b]{0.48\textwidth}
        \centering
        \includegraphics[width=\linewidth]{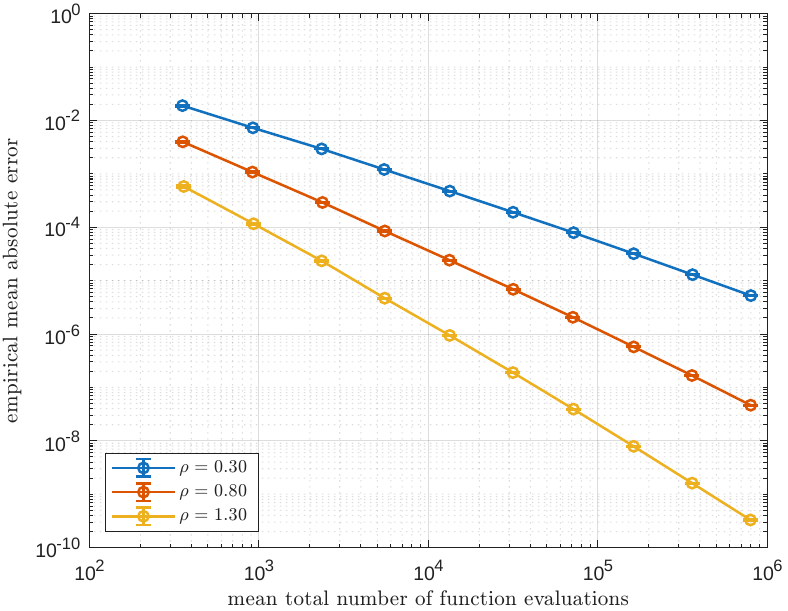}
        \subcaption{$d=1$}
        \label{fig:integration_d1}
    \end{minipage}
    \hfill
    \begin{minipage}[b]{0.48\textwidth}
        \centering
        \includegraphics[width=\linewidth]{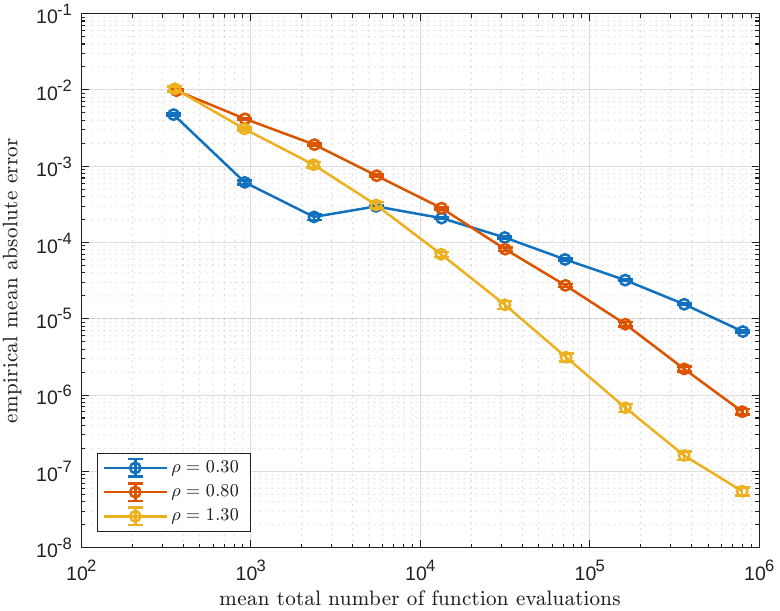}
        \subcaption{$d=4$}
        \label{fig:integration_d4}
    \end{minipage}
    \caption{Empirical mean absolute errors $\widehat E_B(n_{\max})$ versus the mean total number of function evaluations $\overline n_{\mathrm{tot}}$ for the median-of-means tent-transformed rank-1 lattice rule applied to the equal-amplitude test functions $a_j=1/d$. The results correspond to $\rho\in\{0.3,0.8,1.3\}$ and are based on $B=50$ independent realizations.}
    \label{fig:integration_unweighted}
\end{figure}

Figure~\ref{fig:integration_unweighted} shows the results for the low-dimensional equal-amplitude cases $a_j=1/d$ for all $j=1,\ldots,d$ with $d=1$ on the left and $d=4$ on the right. For $d=1$, stable convergence is observed for all three values of $\rho$. The regression over the five largest values of $n_{\max}$ gives the empirical rates $n_{\max}^{-1.2982}, n_{\max}^{-1.8042}, n_{\max}^{-2.2944}$, for $\rho=0.3$, $0.8$, and $1.3$, respectively. These rates agree closely with the theoretical predictions $n_{\max}^{-1.3}, n_{\max}^{-1.8}, n_{\max}^{-2.3}$. For $d=4$, the case $\rho=0.3$ exhibits more pronounced
pre-asymptotic variability, and the observed rates are slightly slower
than in the one-dimensional experiment. The corresponding regression
rates are $n_{\max}^{-1.0148}, n_{\max}^{-1.7789}, n_{\max}^{-2.0516}$.

\begin{figure}[t]
    \centering
    \begin{minipage}[b]{0.48\textwidth}
        \centering
        \includegraphics[width=\linewidth]{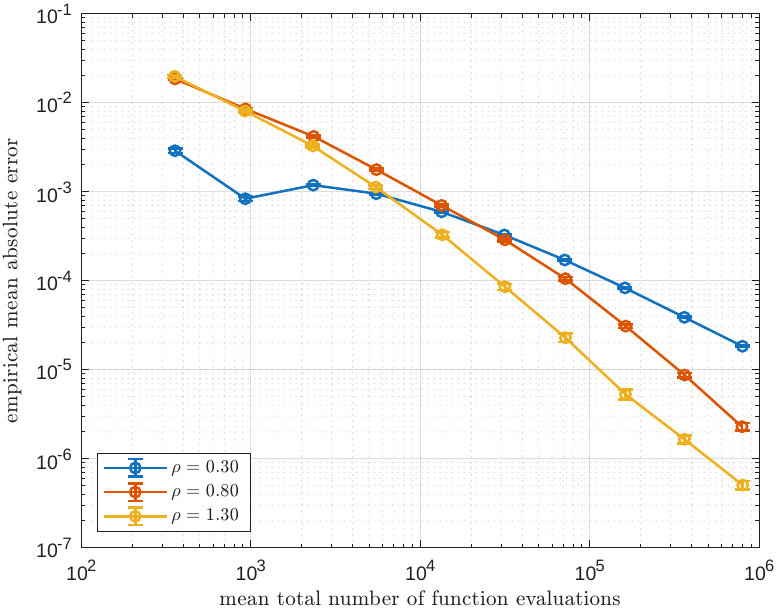}
        \subcaption{$d=20$}
        \label{fig:integration_d20}
    \end{minipage}
    \hfill
    \begin{minipage}[b]{0.48\textwidth}
        \centering
        \includegraphics[width=\linewidth]{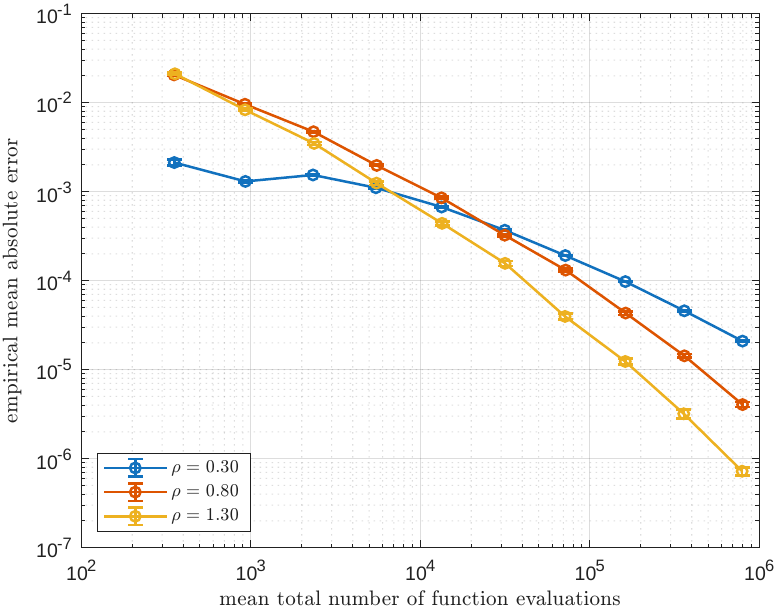}
        \subcaption{$d=100$}
        \label{fig:integration_d100}
    \end{minipage}
    \caption{Empirical mean absolute errors $\widehat E_B(n_{\max})$ versus the mean total number of function evaluations $\overline n_{\mathrm{tot}}$ for the anisotropic test functions $a_j=j^{-2}$. Results are shown for $d=20$ and $d=100$, with $\rho\in\{0.3,0.8,1.3\}$ and $B=50$ independent realizations.}
    \label{fig:integration_weighted}
\end{figure}

Figure~\ref{fig:integration_weighted} presents the results for the
higher-dimensional anisotropic cases $a_j=j^{-2}$ for $j=1,\ldots,d$, with $d=20$ on the left and $d=100$ on the right. The convergence behavior is remarkably similar in the two dimensions. This reflects the rapid decay of the amplitudes $a_j$, which makes the higher-index coordinates progressively less influential. For $d=100$, the regression over the five largest values of $n_{\max}$ gives the empirical rates $n_{\max}^{-1.034}, n_{\max}^{-1.5847}, n_{\max}^{-1.915}$, for $\rho=0.3$, $0.8$, and $1.3$, respectively. Although these pre-asymptotic rates are somewhat slower than the corresponding one-dimensional rates, the experiment still demonstrates stable convergence in a dimension as high as $d=100$ when the coordinate amplitudes decay sufficiently rapidly.

%%%%%%%%%%%%%%%%%%%%%%%%%%%%%%%%%%%%%%%%%%%%%%%%%%
\subsection{Function approximation}
Finally, we investigate the empirical performance of the multiple tent-transformed lattice algorithm \eqref{eq:multiple_lattice_approx}, where construction is described in Algorithm~\ref{alg:multiple_lattice}. We use the same family of test functions $f_{\rho,\bsa}(\bsx)$ as in the preceding numerical integration experiments, with $\rho=3/4$. As noted above, $f_{\rho,\bsa}(\bsx)\in H_{\alpha,\bsgamma,1}^{\sob} \Leftrightarrow \alpha<5/4$. To construct the multiple tent-transformed lattice algorithm $A^{\mathrm{mult}}(f)$, we first define the symmetric weighted hyperbolic cross set by
\begin{align}
    \mathcal{A}_d^{\mathrm{sym}} \coloneqq \mathcal{A}_{d,5/4,\bsa}^{\mathrm{sym}}(N_p) = \left\{\boldsymbol{k} \in \mathbb{Z}^d \mid r_{5/4,\underline{\kappa}_{5/4}^{-1/2}\bsa} (|\boldsymbol{k}|)\le N_p \right\},
\end{align}
where $N_p$ is chosen from the set
\begin{align}
    \mathcal{R} = \{ N_p \mid N_p \text{ is the largest prime less than } 2^k \text{ for } k \in \{5,6,\ldots,17\} \}.
\end{align}
Here, we choose the rescaled weights $\underline{\kappa}_{5/4}^{-1/2}\bsa$, where $\underline{\kappa}_{5/4}^{-1/2}$ is given by
\begin{align*}
    \underline{\kappa}_{5/4}^{-1/2}  = (c'_{5/4})^{-1/2} = \left(\frac{3}{8}\pi^{7/2}\int_{0}^{1/2} \frac{\sin^2(\pi t)}{t^{3/2}} \, \mathrm{d}t\right)^{-1/2} \approx 0.171.
\end{align*}
For each $N_p\in \mathcal{R}$, we repeatedly run Algorithm~\ref{alg:multiple_lattice} (with $c=1.2$ and $\delta=0.001$) until the covering condition $\bigcup_{t=1}^L \Acal_{d,t}^{\sym} = \Acal_d^{\sym}$ is satisfied.

Using the index set $\Acal_d^{\sym}$ and the rescaled weights above, we then construct $A^{\mathrm{mult}}(f)$ in \eqref{eq:multiple_lattice_approx}, requiring $n_{\mathrm{tot}}$ function evaluations. Since the test function belongs to $H_{\alpha,\bsgamma,d}^{\sob}$ for every $\alpha<5/4$, Theorem~\ref{thm:approximation} suggests, in the limit $\alpha \to 5/4^-$, an $L_\infty$ error decay arbitrarily close to $\mathcal{O}(n_{\text{total}}^{-3/4+\epsilon})$ for arbitrarily small $\epsilon>0$. Since the exact $L_\infty$ error cannot be evaluated in practice, we approximate it by the maximum of the pointwise errors $|f_{\rho,\bsa}(\bsx)-A^{\mathrm{mult}}(f)(\bsx)|$ over the first $2^{15}$ points of the $d$-dimensional Sobol’ sequence.

Figure~\ref{fig:approximation} presents the results for the anisotropic choice $a_j=j^{-3}$ for $j=1,\ldots,d$, with $d=2$ in the left panel and $d=100$ in the right panel. In both cases, the observed convergence behavior closely follows the reference rate $\mathcal{O}(n_{\mathrm{total}}^{-0.9\alpha+0.5})$, which is slightly slower than the limiting theoretical rate $\mathcal{O}(n_{\mathrm{total}}^{-\alpha+0.5})$ for $\alpha=5/4$.

\begin{figure}[t]
    \centering
    \begin{minipage}[b]{0.48\textwidth}
        \centering
        \includegraphics[width=\linewidth]{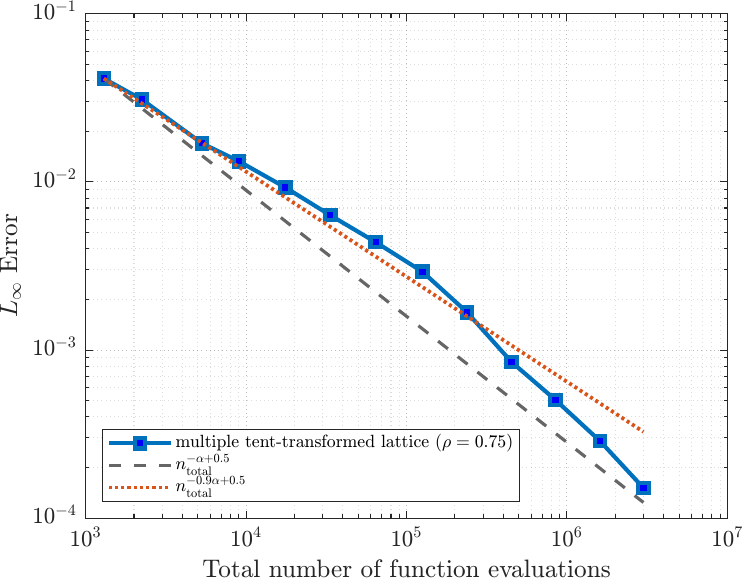}
        \subcaption{$d=2,\;\rho=3/4,\;a_j=j^{-3}$}
        \label{fig:approximation_d2}
    \end{minipage}
    \hfill
    \begin{minipage}[b]{0.48\textwidth}
        \centering
        \includegraphics[width=\linewidth]{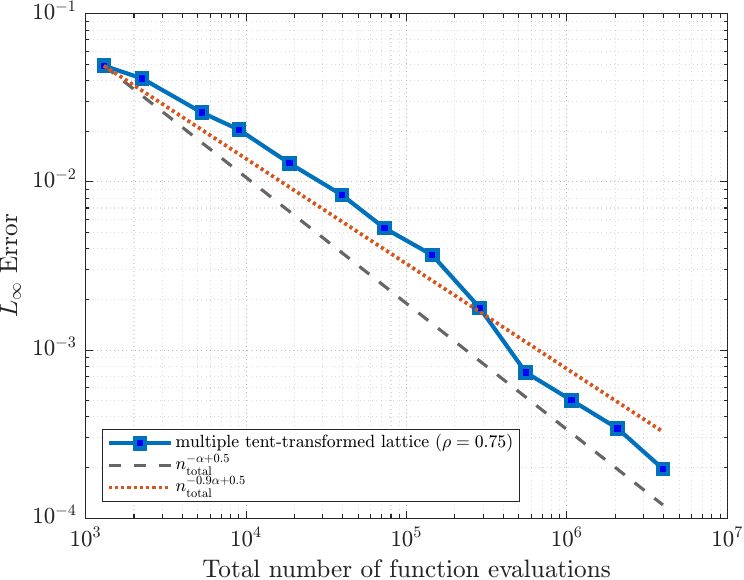}
        \subcaption{$d=100,\; \rho=3/4,\;a_j=j^{-3}$}
        \label{fig:approximation_d100}
    \end{minipage}
    \caption{Empirical discrete $L_{\infty}$ error of $A^{\mathrm{mult}}(f)$ versus the total number of function evaluations $n_{\mathrm{tot}}$ for the anisotropic test functions  $a_j=j^{-3}$. Results are shown for $d=2$ and $d=100$ with $\rho=3/4$.}
    \label{fig:approximation}
\end{figure}

\section*{Acknowledgment}
The authors thank Yannick Meiners (Osnabr\"uck) for a helpful remark on related embedding results among fractional smoothness spaces.

\bibliographystyle{siam}
\bibliography{ref.bib}

\end{document}